\documentclass[10pt]{amsart}
\usepackage{verbatim}
\usepackage{eucal,url,amssymb,stmaryrd,enumerate,amscd}
\usepackage{amsmath}
\usepackage[pagebackref]{hyperref}
\usepackage{amsfonts}
\usepackage{amsthm}

\numberwithin{equation}{section}
\newtheorem{thrm}{Theorem}[section]
\newtheorem{lemma}[thrm]{Lemma}
\newtheorem{prop}[thrm]{Proposition}
\newtheorem{cor}[thrm]{Corollary}
\newtheorem{dfn}[thrm]{Definition}

\newtheorem{rmrk}[thrm]{Remark}

\newtheorem{conv}[thrm]{Convention}
\newcommand{\QH}{\boldsymbol {G\,(\mathbb{H})}}

 1

\begin{document}

\begin{abstract}
We construct
left invariant quaternionic contact (qc)
structures on Lie groups with zero and non-zero torsion and with
non-vanishing quaternionic contact  conformal curvature tensor,
thus showing the existence of non-flat quaternionic contact
manifolds. We prove that the product of the real
line with a seven dimensional manifold, equipped with a certain qc structure,
has a quaternionic K\"ahler metric as well as
a metric with holonomy contained in $Spin(7)$.
As a consequence we determine explicit quaternionic K\"ahler metrics
and $Spin(7)$-holonomy metrics which seem to be { new.}
Moreover, we give
explicit  non-compact eight dimensional almost quaternion
hermitian manifolds with either a closed fundamental four form or fundamental two forms defining a differential ideal that are not quaternionic K\"ahler.
\end{abstract}



\title[Quaternionic K\"ahler and Spin(7) metrics arising from  qc Einstein structures ]
{Quaternionic K\"ahler and Spin(7) metrics arising from  quaternionic contact Einstein structures }
\date{\today }
\thanks{This work has been partially funded by grant MCINN (Spain)
MTM2008-06540-C02-01/02}
\author{L.C. de Andr\'es}
\address[Luis C. de Andr\'es, Marisa Fern\'andez, Jos\'e A. Santisteban]{Universidad del
Pa\'{\i}s Vasco\\
Facultad de Ciencia y Tecnolog\'{\i}a, Departamento de Matem\'aticas\\
Apartado 644, 48080 Bilbao\\
Spain} \email{luisc.deandres@ehu.es} \email{marisa.fernandez@ehu.es}
\email{joseba.santisteban@ehu.es}
\author{M. Fern\'andez}
\author{S. Ivanov}
\address[Stefan Ivanov]{University of Sofia, Faculty of Mathematics and
Informatics, blvd. James Bourchier 5, 1164, Sofia, Bulgaria}
\address{and Department of Mathematics,
University of Pennsylvania, Philadelphia, PA 19104-6395}
\email{ivanovsp@fmi.uni-sofia.bg}
\author{J.A. Santisteban}
\author{L. Ugarte}
\address[Luis Ugarte]{Departamento de Matem\'aticas\,-\,I.U.M.A.\\
Universidad de Zaragoza\\
Campus Plaza San Francisco\\
50009 Zaragoza, Spain} \email{ugarte@unizar.es}
\author{D. Vassilev}
\address[Dimiter Vassilev]{ Department of Mathematics and Statistics\\
University of New Mexico\\
Albuquerque, New Mexico, 87131-0001} \email{vassilev@math.unm.edu}
\maketitle \tableofcontents

\setcounter{tocdepth}{2}

\section{Introduction}

It is well known that the sphere at infinity of a  non-compact
symmetric space $M$ of rank one carries a natural
Carnot-Carath\'eodory structure (see \cite{M,P}). Quaternionic
contact  structures were introduced by Biquard in \cite{Biq1,Biq2},
and they appear naturally as the conformal boundary at infinity of
quaternionic K\"ahler spaces.  Such structures are also relevant
for the quaternionic contact Yamabe problem which is naturally
connected with the extremals and the best constant in an associated
Sobolev-type (Folland-Stein \cite{FS}) embedding on the quaternionic
Heisenberg group \cite{Wei,IMV,IMV1}.

Following Biquard, \cite{Biq1,Biq2}, quaternionic contact
structure (\emph{qc structure}) on a real $(4n+3)$-dimensional
manifold $M$ is a codimension three distribution $H$ locally given
as the kernel of a $\mathbb{R}^3$-valued $1$-form
$\eta=(\eta_1,\eta_2,\eta_3)$, such that, the three $2$-forms
$d\eta_i|_H$ are the (local) fundamental forms of a quaternionic
structure on $H$. The 1-form $\eta$ is determined up to a
conformal factor and the action of $SO(3)$ on $\mathbb{R}^3$, and
therefore $H$ is equipped with a conformal class $[g]$ of
Riemannian metrics. The transformations preserving a given qc
structure $\eta$, i.e. $\bar\eta=\mu\Psi\eta$ for a positive
smooth function $\mu$ and a non-constant $SO(3)$ matrix $\Psi$ are
called \emph{quaternionic contact conformal (qc conformal for
short) transformations}. If the function $\mu$ is constant we have
\emph{qc-homothetic transformations}. To every metric in the fixed
conformal class $[g]$ on $H$ one can associate a linear connection
preserving the qc structure, see \cite{Biq1}, which we shall call
the Biquard connection.
This connection is
invariant under qc homothetic transformations but changes in a
non-trivial way under qc conformal transformations. { The torsion endomorphism of the Biquard connection is the
obstruction for a qc structure to be locally qc homothetic to a 3-Sasakian \cite{IMV,IV1,IMV2} and its vanishing is
equivalent to the vanishing of the trace-free part of the horizontal qc-Ricci
forms, i.e. to the condition that the qc structure is qc Einstein \cite{IMV}.

The quaternionic Heisenberg group $\boldsymbol{G\,(\mathbb{H})}$
with its standard left-invariant qc structure is locally the unique (up to
a $SO(3)$-action) example of a qc structure with flat Biquard
connection \cite{IMV}. { In fact, the vanishing of the curvature  on
$H$ implies the flatness of the Biquard connection \cite{IV}.}
The quaternionic Cayley transform is a qc
conformal equivalence between the standard 3-Sasakian structure on
the $(4n+3)$-dimensional sphere $S^{4n+3}$ minus a point and the
flat qc structure on $\boldsymbol{G\,(\mathbb{H})}$ \cite{IMV}.
All qc structures locally qc conformal to
$\boldsymbol{G\,(\mathbb{H})}$ and $S^{4n+3}$ are characterized in
\cite{IV} by the vanishing of a tensor invariant, the qc-conformal
curvature $W^{qc}$ defined in terms of the curvature and torsion
of the Biquard connection.

Examples of qc manifolds  arising from quaternionic K\"ahler
deformations are given in \cite{Biq1,Biq2,D1}.
Duchemin shows \cite{D1} that for any qc manifold there exists a
quaternionic K\"ahler manifold such that the qc manifold is realized
as a hypersurface. However, the embedding in his construction is not
isometric and it is difficult to write an explicit expression of the
quaternionic K\"ahler metric except the 3-Sasakian case where the
cone metric is hyperK\"ahler.

One purpose of this paper is to find new explicit examples of   qc
structures. We construct explicit left invariant qc structures on
 three Lie groups  of dimension seven ( that we call $L_1$, $L_2$ and $L_3$)
with zero and non-zero torsion endomorphism of the
Biquard connection for which the qc-conformal curvature tensor does
not vanish, $W^{qc}\not=0$, thus showing the existence of qc
manifolds not locally qc conformal to the quaternionic Heisenberg
group $\boldsymbol{G\,(\mathbb{H})}$. We present a left invariant qc
structure with zero torsion endomorphism of the Biquard connection on a seven
dimensional non-nilpotent Lie group $L_1$. Surprisingly, we obtain
that this qc structure is locally qc conformal to the flat qc
structure on the two-step nilpotent quaternionic Heisenberg group
$\boldsymbol{G\,(\mathbb{H})}$ showing that the qc conformal
curvature is zero and applying the main result in \cite{IV}.
Consequently, this fact yields the existence of a local function
$\mu$ such that the qc conformal transformation $\bar\eta=\mu\eta$
preserves  the vanishing of the torsion of the Biquard connection.

The second goal of the paper is to construct explicit quaternionic
K\"ahler and $Spin(7)$-holonomy metrics, i.e. metrics with
holonomy contained in $Sp(n)Sp(1)$ and $Spin(7)$, respectively, on a product of
a qc manifold with a real line. We generalize the notion of a qc
structure, namely, we define  an $Sp(n)Sp(1)$-hypo structure on a
$(4n+3)$-dimensional manifold as the structure induced on an orientable hypersurface
of a quaternionic K\"ahler manifold. { We show that a qc
structure is an $Sp(n)Sp(1)$-hypo precisely when its fundamental
four form, defined by (\ref {fund4}), vanishes.}
 In Theorem~\ref{qkmetric}, we prove that there is a quaternionic K\"ahler structure
on the product of the real line with a smooth qc Einstein manifold
of dimension bigger than seven.  In dimension seven, we show that
the product of a real line with a qc Einstein structure of
constant qc scalar curvature has a quaternionic K\"ahler metric
(Theorem~\ref{qkmetric}) { as well as} a metric with holonomy
contained in $Spin(7)$ (Theorem~\ref{spin7metric}). The construction and the properties of the obtained metrics with
special holonomy depends on the sign (or the vanishing) of the qc
scalar curvature. { In the  negative}
qc scalar curvature case, we present explicit quaternionic
K\"ahler metrics, { complete as sub-Riemannian metrics,} and
$Spin(7)$-holonomy metrics on the product of the seven dimensional
Lie groups $L_1$ and $L_2$ with the real line, some of which seem
to be new. { In the case of zero qc scalar curvature, using
the} quaternionic Heisenberg group, we rediscover the complete
Einstein metric on an eight-dimensional solvable Lie group
constructed  by Gibbons at al in \cite{GLPS} as an Einstein metric
starting with a $T^3$ bundle over $T^4$, \cite[equation
(148)]{GLPS}. Thus, we show that the Einstein metric in dimension
eight discovered in \cite{GLPS} is in fact a quaternionic K\"ahler
metric  and extend it to obtain complete quaternionic K\"ahler
metrics on a $4n+4$ dimensional solvable Lie groups constructed on
$\boldsymbol{G\,(\mathbb{H})}\times \mathbb R$.
 In the positive qc scalar curvature case we give a general construction which includes well
 known metrics on the product of a 3-Sasakian manifold with the real line.

It is well known that in dimension eight { an almost quaternion
hermitian} structure with closed fundamental  four form is not necessarily quaternionic K\"ahler \cite{Sw}. This fact was confirmed
by Salamon constructing in \cite{Sal} a compact example of { an
almost quaternion hermitian} manifold with closed fundamental four
form which is not Einstein, and therefore it is not a quaternionic
K\"ahler. We give a three parameter family of explicit
non-compact eight dimensional { almost quaternion hermitian
manifolds} with closed fundamental four form which are not
quaternionic K\"ahler. We also check that these examples are not
Einstein.

To the best of our knowledge there is no {known example of  an
almost quaternion hermitian}  eight dimensional manifold with closed
fundamental four form which is Einstein but not quaternionic
K\"ahler.

Finally, we obtain an explicit family of almost quaternion
hermitian  structures such that the fundamental 2-forms  define a differential ideal but the structure is
not a quaternionic K\"{a}hler manifold, see also \cite{Macia} for earlier examples.

\begin{conv}
\label{conven} \hfill\break\vspace{-15pt}
\begin{enumerate}
\item[a)] We shall use $X,Y,Z,U$ to denote horizontal vector fields, i.e.
$X,Y,Z,U\in H$;
\item[b)] $\{e_1,\dots,e_{4n}\}$ denotes a  local orthonormal basis of the horizontal
space $H$;
\item[c)] The triple $(i,j,k)$ denotes any cyclic permutation of $(1,2,3)$.
\item[d)] $s$  will be any number from the set $\{1,2,3\}, \quad
s\in\{1,2,3\}$.
\end{enumerate}
\end{conv}

\textbf{Acknowledgments} We thank  Charles Boyer for useful
conversions leading to Remark \ref{r:CBoyer}.

The research was initiated during the visit of the third author to
the Abdus Salam ICTP, Trieste as a Senior Associate, Fall 2008. He
also thanks ICTP for providing the support and an excellent research
environment. S.I. is partially supported by the Contract 082/2009
with the University of Sofia `St.Kl.Ohridski'. S.I and D.V. are
partially supported by Contract ``Idei", DO 02-257/18.12.2008 and DID 02-39/21.12.2009. This
work has been also partially supported through grant MCINN (Spain)
MTM2008-06540-C02-01/02.

\section{Quaternionic contact manifolds}

In this section we will briefly review the basic notions of
quaternionic contact geometry and recall some results from
\cite{Biq1}, \cite{IMV} and \cite{IV} which we will use in this
paper.

\subsection{Quaternionic contact structures and the Biquard connection}

A quaternionic contact (qc) manifold $(M, g, \mathbb{Q})$ is a
$4n+3$-dimensional manifold $M$ with a codimension three
distribution $H$  locally given as the kernel of a 1-form
$\eta=(\eta_1,\eta_2,\eta_3)$ with values in $\mathbb{R}^3$. In addition $H$ has an $Sp(n)Sp(1)$ structure, that is, it is
equipped with a Riemannian metric $g$ and a rank-three bundle
$\mathbb Q$ consisting of
endomorphisms of $H$ locally generated
by three almost complex structures $I_1,I_2,I_3$ on $H$ satisfying
the identities of the imaginary unit quaternions,
$I_1I_2=-I_2I_1=I_3, \quad I_1I_2I_3=-id_{|_H}$ which are
hermitian compatible with the metric $g(I_s.,I_s.)=g(.,.)$ and the following compatibility condition holds
$\qquad 2g(I_sX,Y)\ =\ d\eta_s(X,Y), \quad X,Y\in H.$

A special phenomena, noted in \cite{Biq1}, is that the contact
form $\eta$ determines the  quaternionic structure and the metric on
the horizontal distribution in a unique way.

Correspondingly, given a qc manifold we shall denote with
$\eta$ any associated contact form. The associated contact form is
determined up to  an $SO(3)$-action, namely if $\Psi\in SO(3)$ then
$\Psi\eta$ is again a contact form satisfying the above
compatibility condition (rotating also the almost complex
structures). On the other hand, if we consider the conformal class
$[g]$ on $H$, the associated contact forms are determined up to a
multiplication with a positive conformal factor $\mu$ and an
$SO(3)$-action, namely if $\Psi\in SO(3)$ then $\mu\Psi\eta$ is a
contact form associated with a metric in the conformal class $[g]$
on $H$. A qc manifold $(M, \bar g,\mathbb{Q} )$
is called qc conformal
to $(M, g,\mathbb{Q} )$ if $\bar g\in [g]$. In that case, if $\bar\eta$ is a
corresponding associated $1$-form with complex structures $\bar I_s$,
$s=1,2,3,$ we have $\bar\eta\ =\ \mu\, \Psi\,\eta$ for some $\Psi\in
SO(3)$ with smooth functions as entries and a positive function
$\mu$. In particular, starting with a qc manifold $(M, \eta)$ and
defining $\bar\eta\ =\ \mu\, \eta$ we obtain a qc manifold $(M,
\bar\eta)$ qc conformal to the original one.

If the first Pontryagin class of $M$ vanishes then the 2-sphere bundle of
$\mathbb{R}^3$-valued 1-forms is trivial \cite{AK}, i.e. there is a
globally defined form $\eta$ that anihilates $H$, we denote the
corresponding qc manifold $(M,\eta)$. In this case the 2-sphere of
associated almost complex structures is also globally defined on
$H$.

On a qc manifold with a fixed metric $g$ on $H$ there exists a
canonical connection defined in \cite{Biq1} when the dimension $(4n+3)>7$,
and in \cite{D} for the 7-dimensional case. We have
\begin{thrm}\label{biqcon}.\cite{Biq1} {Let $(M, g,\mathbb{Q})$ be a qc
manifold} of dimension $4n+3>7$ and a fixed metric $g$ on $H$ in
the conformal class $[g]$. Then there exists a unique connection
$\nabla$ with
torsion $T$ on $M^{4n+3}$ and a unique supplementary subspace $V$ to $H$ in
$TM$, such that:
\begin{enumerate}
\item[i)]
$\nabla$ preserves the decomposition $H\oplus V$ and the $Sp(n)Sp(1)$ structure on $H$,
i.e. $\nabla g=0,  \nabla\sigma \in\Gamma(\mathbb Q)$ for a section
$\sigma\in\Gamma(\mathbb Q)$, and its torsion on $H$ is given by $T(X,Y)=-[X,Y]_{|V}$;
\item[ii)] for $\xi\in V$, the endomorphism $T(\xi,.)_{|H}$ of $H$ lies in
$(sp(n)\oplus sp(1))^{\bot}\subset gl(4n)$;
\item[iii)]
the connection on $V$ is induced by the natural identification $\varphi$ of
$V$ with the subspace $sp(1)$ of the endomorphisms of $H$, i.e.
$\nabla\varphi=0$.
\end{enumerate}
\end{thrm}
In ii), the inner product $<,>$ of $End(H)$ is given by $<A,B> = {
\sum_{i=1}^{4n} g(A(e_i),B(e_i)),}$ for $A, B \in End(H)$.

We shall call the above connection \emph{the Biquard connection}.
Biquard \cite{Biq1} also described the supplementary subspace $V$, namely
$V$ is (locally) generated by vector fields $\{\xi_1,\xi_2,\xi_3\}$,
such that
\begin{equation}  \label{bi1}
\begin{aligned} \eta_s(\xi_k)=\delta_{sk}, \qquad (\xi_s\lrcorner
d\eta_s)_{|H}=0,\\ (\xi_s\lrcorner d\eta_k)_{|H}=-(\xi_k\lrcorner
d\eta_s)_{|H}, \end{aligned}
\end{equation}
where $\lrcorner$ denotes the interior multiplication. The vector
fields $\xi_1,\xi_2,\xi_3$ are called Reeb
vector fields.

If the dimension of $M$ is seven, there might be no vector fields
satisfying \eqref{bi1}.  Duchemin shows in \cite{D} that if we
assume, in addition, the existence of Reeb vector fields as in
\eqref{bi1}, then Theorem~\ref{biqcon}
holds. Henceforth, by a qc structure in dimension $7$ we shall
mean a qc structure satisfying \eqref{bi1}.

Notice that equations \eqref{bi1} are invariant under the natural
$SO(3)$
action. Using the triple of Reeb vector fields we extend $g$ to a metric on
$M$ by requiring
$span\{\xi_1,\xi_2,\xi_3\}=V\perp H \text{ and }
g(\xi_s,\xi_k)=\delta_{sk}.
$
\hspace{2mm} \noindent The extended metric does not depend on the action of
$SO(3)$ on $V$, but it changes in an obvious manner if $\eta$ is
multiplied by a conformal factor. Clearly, the Biquard connection
preserves the extended metric on $TM, \nabla g=0$.
Since the Biquard connection is metric it is connected with the Levi-Civita connection
$\nabla^g$ of the metric $g$ by the general formula
\begin{equation}  \label{lcbi}
g(\nabla_AB,C)=g(\nabla^g_AB,C)+\frac12\Big[
g(T(A,B),C)-g(T(B,C),A)+g(T(C,A),B)\Big].
\end{equation}

The covariant derivative of the qc structure with respect to the
Biquard connection and the covariant derivative of the distribution
$V $ are given by $
\nabla I_i=-\alpha_j\otimes I_k+\alpha_k\otimes I_j,\quad
\nabla\xi_i=-\alpha_j\otimes\xi_k+\alpha_k\otimes\xi_j.$
The $sp(1)$-connection 1-forms $\alpha_s$ on $H$ are expressed in
\cite{Biq1} by
\begin{gather}  \label{coneforms}
\alpha_i(X)=d\eta_k(\xi_j,X)=-d\eta_j(\xi_k,X), \quad X\in H, \quad
\xi_i\in V,
\end{gather}
while the $sp(1)$-connection 1-forms $\alpha_s$ on the vertical
space $V$ are calculated in \cite{IMV}
\begin{gather}  \label{coneform1}
\alpha_i(\xi_s)\ =\ d\eta_s(\xi_j,\xi_k) -\
\delta_{is}\left(\frac{S}2\ +\ \frac12\,\left(\,
d\eta_1(\xi_2,\xi_3)\ +\ d\eta_2(\xi_3,\xi_1)\ + \
d\eta_3(\xi_1,\xi_2)\right)\right),
\end{gather}
where  $S$ is the \emph{normalized} qc scalar curvature defined
below in \eqref{qscs}. The vanishing of the $sp(1)$-connection
$1$-forms on $H$ implies the vanishing of the torsion endomorphism of
the Biquard connection (see \cite{IMV}).

The fundamental 2-forms $\omega_i, i=1,2,3$ \cite{Biq1}  are defined
by
$2\omega_{i|H}\ =\ \, d\eta_{i|H},\quad \xi\lrcorner\omega_i=0,\quad
\xi\in V.
$
The properties of the Biquard connection are encoded in the
properties of
the torsion
endomorphism $T_{\xi}=T(\xi,\cdot) : H\rightarrow H, \quad \xi\in V$. Decomposing
the endomorphism $T_{\xi}\in(sp(n)+sp(1))^{\perp}$
into its symmetric part $T^0_{\xi}$ and skew-symmetric part
$b_{\xi}, T_{\xi}=T^0_{\xi} + b_{\xi} $, O. Biquard shows in
\cite{Biq1} that the torsion $T_{\xi}$ is completely trace-free,
$tr\, T_{\xi}=tr\, T_{\xi}\circ
I_s=0$, its symmetric part has the properties
$T^0_{\xi_i}I_i=-I_iT^0_{\xi_i}\quad
I_2(T^0_{\xi_2})^{+--}=I_1(T^0_{\xi_1})^{-+-},\quad
I_3(T^0_{\xi_3})^{-+-}=I_2(T^0_{\xi_2})^{--+},\quad
I_1(T^0_{\xi_1})^{--+}=I_3(T^0_{\xi_3})^{+--} $, where
where the upperscript $+++$ means commuting  with all three $I_i$, $+--$
indicates commuting with $I_1$ and anti-commuting with the other two and etc.
The skew-symmetric part can be represented as $b_{\xi_i}=I_iu$, where
$u$ is a traceless symmetric (1,1)-tensor on $H$ which commutes with
$I_1,I_2,I_3$. If $n=1$ then the tensor $u$ vanishes identically,
$u=0$ and the torsion is a symmetric tensor, $T_{\xi}=T^0_{\xi}$.

Any 3-Sasakian manifold has zero torsion endomorphism, and
 the converse is true if in addition the qc scalar curvature (see
\eqref{qscs}) is a positive constant \cite{IMV}. We remind that a
$(4n+3)$-dimensional  Riemannian manifold $(M,g)$ is
called 3-Sasakian if the cone
metric $g_c=t^2g+dt^2$ on $C=M\times \mathbb{R}^+$ is a
hyper K\"ahler metric, namely, it has holonomy contained
in $Sp(n+1)$ \cite{BGN}. A 3-Sasakian manifold of dimension
$(4n+3)$ is Einstein with positive Riemannian scalar curvature
$(4n+2)(4n+3)$ \cite{Kas} and if complete it is compact with a
finite fundamental group, (see \cite{BG} for a nice overview of
3-Sasakian spaces).

\subsection{Torsion and curvature}

Let $R=[\nabla,\nabla]-\nabla_{[\ ,\ ]}$ be the curvature tensor of
$\nabla$ and the dimension is $4n+3$. We denote the curvature tensor
of type (0,4) by the same letter, $R(A,B,C,D):=g(R(A,B)C,D)$,
$A,B,C,D \in \Gamma(TM)$. The Ricci $2$-forms  and the normalized
scalar curvature of the Biquard connection, called \emph{qc-Ricci
forms} $\rho_s$ and \emph{normalized qc-scalar curvature} $S$, respectively, are
defined by
\begin{equation}  \label{qscs}
4n\rho_s(A,B)=R(A,B,e_a,I_se_a), \quad 8n(n+2)S=R(e_b,e_a,e_a,e_b).
\end{equation}
The $sp(1)$-part of $R$ is determined by the Ricci 2-forms and the
connection 1-forms by
\begin{equation}  \label{sp1curv}
R(A,B,\xi_i,\xi_j)=2\rho_k(A,B)=(d\alpha_k+\alpha_i\wedge\alpha_j)(A,B),
\qquad A,B \in \Gamma(TM).
\end{equation}
The two $Sp(n)Sp(1)$-invariant trace-free symmetric 2-tensors
$T^0(X,Y)=
g((T_{\xi_1}^{0}I_1+T_{\xi_2}^{0}I_2+T_{ \xi_3}^{0}I_3)X,Y)$, $U(X,Y)
=g(uX,Y)$ on $H$, introduced in \cite{IMV}, have the properties:
\begin{equation}  \label{propt}
\begin{aligned} T^0(X,Y)+T^0(I_1X,I_1Y)+T^0(I_2X,I_2Y)+T^0(I_3X,I_3Y)=0, \\
U(X,Y)=U(I_1X,I_1Y)=U(I_2X,I_2Y)=U(I_3X,I_3Y). \end{aligned}
\end{equation}
In dimension seven $(n=1)$, the tensor $U$ vanishes identically,
$U=0$.

We shall need the following identity taken from
\cite[Proposition~2.3]{IV}
\begin{equation}  \label{need}
4g(T^0(\xi_s,I_sX),Y)=T^0(X,Y)-T^0(I_sX,I_sY)
\end{equation}
\begin{dfn}A qc structure is said to be qc Einstein if the horizontal qc-Ricci 2-forms are scalar multiple of the fundamental 2-forms,
$$\rho_s(X,Y)=\nu_s\omega_s(X,Y).$$
\end{dfn}
For a qc Einstein structure the functions $\nu_s$ are all equal and can be expressed as a constant multiple of the qc scalar curvature \cite{IMV}.

The horizontal Ricci 2-forms can be expressed in terms of the
torsion of the Biquard connection \cite{IMV} (see also
\cite{IMV1,IV}). We collect the necessary facts from
\cite[Theorem~1.3, Theorem~3.12, Corollary~3.14, Proposition~4.3 and
Proposition~4.4]{IMV} with slight modification presented in
\cite{IV}

\begin{thrm}\label{sixtyseven}
.\cite{IMV} On a $(4n+3)$-dimensional qc manifold $(M,\eta,\mathbb{Q})$ the next
formulas hold
\begin{equation} \label{sixtyfour}
\begin{aligned} \rho_l(X,I_lY) \ & =\
-\frac12\Bigl[T^0(X,Y)+T^0(I_lX,I_lY)\Bigr]-2U(X,Y)-Sg(X,Y),\\ T(\xi_{i},\xi_{j})& =-S\xi_{k}-[\xi_{i},\xi_{j}]_{H}, \qquad S\  =\
-g(T(\xi_1,\xi_2),\xi_3)
\\ g(T(\xi_i,\xi_j),X) &
=-\rho_k(I_iX,\xi_i)=-\rho_k(I_jX,\xi_j)=-g([\xi_i,\xi_j],X),\qquad
\rho_i(\xi_i,\xi_j)+\rho_k(\xi_k,\xi_j)=\frac12\xi_j(S);\\
\rho_{i}(X,\xi_{i})\ & =\ -\frac {X(S)}{4} \ +\ \frac
12\, \left
(-\rho_{i}(\xi_{j},I_{k}X)+\rho_{j}(\xi_{k},I_{i}X)+\rho_{k}(\xi_{i},I_{j}X)
\right). \end{aligned}
\end{equation}
For $n=1$ the above formula holds with $U=0$.

The qc Einstein condition
is equivalent to the vanishing of the torsion endomorphism of the
Biquard connection. In this case $S$ is constant and  the vertical distribution is integrale provided $n>1$.
\end{thrm}

\subsection{The qc conformal curvature}

The qc conformal curvature tensor $W^{qc}$, introduced  in \cite{IV},
is the obstruction for a qc structure to be locally qc
conformal to the flat structure on the quaternionic Heisenberg group $
\boldsymbol{G\,(\mathbb{H})}$. Denote $L_0=\frac12T^0+U$, the tensor $W^{qc}$ can be expressed by \cite{IV}
\begin{multline}  \label{qcwdef1}
W^{qc}(X,Y,Z,V) =
R(X,Y,Z,V)+ (g\owedge L_0))(X,Y,Z,V)+
\sum_{s=1}^3(\omega_s%
\owedge I_sL_0)(X,Y,Z,V) \\
-\frac12\sum_{s=1}^3\Bigl[\omega_s(X,Y)\Bigl\{T^0(Z,I_sV)-T^0(I_sZ,V)\Bigr\}
+ \omega_s(Z,V)\Bigl\{T^0(X,I_sY)-T^0(I_sX,Y)-4U(X,I_sY)\Bigr\}\Bigr] \\
+\frac{S}4\Big[(g\owedge g)(X,Y,Z,V)+\sum_{s=1}^3\Bigl((\omega_s\owedge
\omega_s)(X,Y,Z,V) +4\omega_s(X,Y)\omega_s(Z,V)\Bigr) \Big],
\end{multline}
where $I_s U\, (X,Y) = -U (X,I_s Y)$ and $\owedge$ is the
Kulkarni-Nomizu product of (0,2) tensors, for example,
\begin{multline*}
(\omega_s\owedge U)(X,Y,Z,V):=\omega_s(X,Z)U(Y,V)+
\omega_s(Y,V)U(X,Z)-\omega_s(Y,Z)U(X,V)-\omega_s(X,V)U(Y,Z).
\end{multline*}
The main result from \cite{IV} can be stated as follows

\begin{thrm}
\label{main1} .\cite{IV} A qc structure on a $(4n+3)$-dimensional
smooth manifold is locally qc conformal to the standard flat qc
structure on the quaternionic Heisenberg group
$\boldsymbol{G\,(\mathbb{H})}$ if and only if the qc conformal
curvature vanishes, $W^{qc}=0$. In this case, we call the qc
structure a qc conformally flat structure.
\end{thrm}
A qc conformally flat structure is also locally qc conformal to
the standard 3-Sasaki sphere due to the local qc conformal equivalence of
the standard 3-Sasakian structure on the $4n+3$-dimensional sphere
and the quaternionic Heisenberg group \cite{IMV,IV}.

\section{Explicit examples of quaternionic contact structures }

In this section we give explicit examples of qc structures in
dimension seven satisfying the compatibility conditions \eqref{bi1}.
The first example has zero torsion and is locally qc conformal to
the quaternionic Heisenberg group. The second example has zero
torsion while the third is with non-vanishing torsion, and both are
not locally qc conformal to the
quaternionic Heisenberg group. We remind that the zero torsion qc structures are precisely the qc Einstein structures, cf. Theorem \ref{sixtyseven}.

\begin{rmrk}\label{r:CBoyer}
We note explicitly that the vanishing of the torsion endomorphism
implies that, locally,  the structure is homothetic to a 3-Sasakian
structure if the qc scalar curvature is positive. In the seven
dimensional examples below the qc scalar curvature is a negative
constant. In that respect, as pointed by Charles Boyer, there are no
compact invariant with respect to translations 3-Sasakian Lie groups
of dimension seven.
\end{rmrk}

\subsection{Example 0: The quaternionic Heisenberg group ${G\,(\mathbb{H})}$ - the Biquard-flat qc structure}

As a manifold $\boldsymbol{G\,(\mathbb{H})} \ =\mathbb{H}
^n\times\text {Im}\, \mathbb{H}$, while the group multiplication is
given by $( q^{\prime }, \omega^{\prime })\ =\ (q_o,
\omega_o)\circ(q, \omega)\ =\
(q_o\ +\ q, \omega\ +\ \omega_o\ + \ 2\ \text {Im}\ q_o\, \bar q)$, where $
q,\ q_o\in\mathbb{H}^n$ and $\omega, \omega_o\in \text {Im}\,
\mathbb{H}$. The standard flat qc structure is defined by the
left-invariant qc form $\tilde\Theta\ =\ (\tilde\Theta_1,\
\tilde\Theta_2, \ \tilde\Theta_3)\ =\ \frac 12\ (d\omega \ - \
q^{\prime
}\cdot d\bar q^{\prime }\ + \ dq^{\prime }\, \cdot\bar q^{\prime })$, where
$.$ denotes the quaternion multiplication. As a Lie group it can be
characterized by the following structure equations. Denote by ${e^a,
1 \leq a\leq (4n+3)}$ the basis of the left invariant 1-forms, {and
by $e^{ij}$ the wedge product $e^i \wedge e^j$.} The
$(4n+3)$-dimensional quaternionic Heisenberg Lie algebra is the
2-step nilpotent Lie algebra defined by:
\begin{equation}  \label{4n+3heis}
\begin{aligned}
& de^a =0, \qquad 1\leq a \leq 4n, \\
& d\eta_1=de^{4n+1}= 2( e^{12} + e^{34}+\cdots + e^{(4n-3)(4n-2)}+
e^{(4n-1)4n} )= 2\omega_1,
\\
& d\eta_2=de^{4n+2}= 2 (e^{13} +e^{42}+\cdots + e^{(4n-3)(4n-1)}+
e^{4n(4n-2)})= 2\omega_2,
\\
& d\eta_3 = de^{4n+3}= 2 (e^{14} + e^{23}+\cdots + e^{(4n-3)4n}+
e^{(4n-2)(4n-1)})= 2\omega_3.
\end{aligned}
\end{equation}
The Biquard connection
coincides with the flat left-invariant  connection on $\QH$.
This flat qc structure on the quaternionic Heisenberg group $\boldsymbol{G\,(\mathbb{H})}$ is (locally) the unique qc structure with flat Biquard connection \cite{IMV,IV1}. By a rotation of the 1-forms defining the horizontal space of $\boldsymbol{G\,(\mathbb{H})}$ we obtain an equivalent qc-structure (with the same Biquard connection). It is possible to introduce a different not two step nilpotent group structure on  $\mathbb{H}
^n\times\text {Im}\, \mathbb{H}$ with respect to which the rotated forms are left invariant (but not parallel!). Following is an explicit description of this construction in dimension seven.

Consider the seven dimensional quaternionic Heisenberg group. Since $e^4$ is closed we can write $e^4=dx_4$, where $x_4$ is a global function on the manifold $\mathbb{H}\times\text {Im}\, \mathbb{H}$.
Now we can use this function to define a non-left-invariant qc structure on this manifold as follows.
For each $c\in \mathbb{R}$, let
\begin{equation}  \label{ex01rotation}
\begin{aligned}
& \gamma^1= e^1, \quad \gamma^2= \sin(-cx_4)\, e^2 + \cos(-cx_4)\, e^3,\quad \gamma^3= -\cos(-cx_4)\, e^2 + \sin(-cx_4)\, e^3,\quad \gamma^4= e^4,\\
& \gamma^5= \sin(-cx_4)\, e^5 + \cos(-cx_4)\, e^6,\quad \gamma^6= -\cos(-cx_4)\, e^5 + \sin(-cx_4)\, e^6,\quad \gamma^7= e^7.
\end{aligned}
\end{equation}
A direct calculation shows that for  $c\not=0$ the forms $\{ \gamma^l,\ 1 \leq l\leq 7\}$ define a
unique Lie algebra $\mathfrak {l_0}$ with the following structure equations
\begin{equation}  \label{ex01}
\begin{aligned}
&d\gamma^1=0,\quad d\gamma^2=-c\gamma^{34},\quad d\gamma^3=c\gamma^{24},\quad d\gamma^4=0,\\
&d\gamma^5=2\gamma^{12}+2\gamma^{34}+c\gamma^{46},\quad
d\gamma^6=2\gamma^{13}+2\gamma^{42}-c\gamma^{45},\quad
d\gamma^7=2\gamma^{14}+2\gamma^{23}.
\end{aligned}
\end{equation}
{In particular, $l_0$ is an indecomposable solvable Lie algebra.}
Let $e_l, 1 \leq l\leq 7$ be the left invariant vector fields dual
to the 1-forms ${\gamma^l, 1 \leq l\leq 7}$. The
(global) flat qc structure on $\mathbb{H}\times\text {Im}\, \mathbb{H}$ can also be described as follows
$\eta_1=\gamma^5, \quad \eta_2=\gamma^6, \quad \eta_3=\gamma^7, \quad
H=span\{\gamma^1,\dots, \gamma^4\}$, $\omega_1=\gamma^{12}+\gamma^{34}, \quad
\omega_2=\gamma^{13}+\gamma^{42}, \quad \omega_3=\gamma^{14}+\gamma^{23}$.
It is straightforward to check from \eqref{ex01} that the vector
fields
$\xi_1=e_5$, $\xi_2=e_6$, $\xi_3=e_7 $
satisfy the Duchemin compatibility conditions \eqref{bi1} and
therefore the Biquard connection exists and $\xi_s$ are the Reeb
vector fields.

Let $(L_0,\eta,\mathbb{Q})$ be the simply connected connected Lie
group with Lie algebra $\mathfrak l_0$ equipped with the left invariant qc structure
$(\eta,\mathbb{Q})$ defined above. Then, as a consequence of the above construction,
the torsion endomorphism and the curvature of the Biquard connection
are identically zero  but the basis $\gamma_1,\dots,\gamma_7$ is not parallel.  The $Sp(1)$-connection 1-forms  in the basis $\gamma^1, \dots, \gamma^7$ are given by
$
\alpha_1=0, \quad \alpha_2=0, \quad \alpha_3=c\gamma^4.
$

\subsection{Example 1: zero torsion qc-flat structure.}
Denote $\{\tilde e^l,\ 1 \leq l\leq 7\}$ the basis of the left
invariant 1-forms and consider the simply connected connected Lie group $L_1$ with
{indecomposable} Lie algebra $\mathfrak {l_1}$ defined by the following equations
\begin{equation}  \label{ex11}
\begin{aligned}
&de^1=0,\quad de^2=-e^{12}-2e^{34}-\frac12e^{37}+\frac12e^{46},\\
&de^3=-e^{13}+2e^{24}+\frac12e^{27}-\frac12e^{45},\quad
de^4=-e^{14}-2e^{23}-\frac12e^{26}+\frac12e^{35},\\
&de^5=2e^{12}+2e^{34}-\frac12e^{67},\quad
de^6=2e^{13}+2e^{42}+\frac12e^{57},\quad
de^7=2e^{14}+2e^{23}-\frac12e^{56}.
\end{aligned}
\end{equation}
Let $e_l, 1 \leq l\leq 7$ be the left invariant vector field dual
to the 1-forms ${e^l, 1 \leq l\leq 7}$, respectively. A global qc structure on $L_1$ is defined by
\begin{equation}\label{qc1}
\begin{aligned} &\eta_1=e^5, \quad \eta_2=e^6, \quad \eta_3=e^7, \quad
H=span\{e^1,\dots, e^4\},\\ &\omega_1=e^{12}+e^{34}, \quad
\omega_2=e^{13}+e^{42}, \quad \omega_3=e^{14}+e^{23}. \end{aligned}
\end{equation}
It is straightforward to check from \eqref{ex11} that the vector
fields
$\xi_1=e_5$, $\xi_2=e_6$, $\xi_3=e_7 $
satisfy the Duchemin compatibility conditions \eqref{bi1} and
therefore the Biquard connection exists and $\xi_s$ are the Reeb
vector fields.

\begin{thrm}
\label{m1} Let $(L_1,\eta,\mathbb{Q})$ be the simply connected Lie
group
with Lie algebra $\mathfrak {l_1}$ equipped with the left invariant qc structure
$(\eta,\mathbb{Q})$ defined above. Then

\begin{itemize}
\item[a)] The qc structure is qc Einstein 
the normalized qc scalar curvature is a negative constant, $S=-\frac12$.

\item[b)] The qc conformal curvature is zero, $W^{qc}=0$, and therefore
$(L_1,\eta,\mathbb{Q})$ is locally qc conformally flat.
\end{itemize}
\end{thrm}

\begin{proof}
We compute first the connection 1-forms and the horizontal Ricci forms of
the Biquard connection. The Lie algebra structure equations
\eqref{ex11} together with \eqref{coneforms}, \eqref{coneform1} and
\eqref{sp1curv} imply
\begin{equation}  \label{ex1conf1}
\alpha_s=(\frac14-\frac{S}2)\eta_s, \qquad
\rho_s(X,Y)=\frac12d\alpha_s(X,Y)=(\frac14-\frac{S}2)\omega_s(X,Y).
\end{equation}
Equation \eqref{ex1conf1} and \eqref{sixtyfour} allow us to conclude that the
torsion endomorphism is zero and the  normalized qc scalar $S=-\frac12$. Now,
Theorem \ref{sixtyseven} completes the proof of part a).

In view of Theorem~\ref{main1}, to prove part b) we have to show
$W^{qc}=0$. Since the torsion of the
Biquard connection vanishes and $S=-\frac12$, \eqref{qcwdef1} takes
the form
\begin{multline}  \label{qcwdef2}
W^{qc}(X,Y,Z,V)= R(X,Y,Z,V) \\
-\frac{1}8\Big[(g\owedge g)(X,Y,Z,V)+\sum_{s=1}^3\Bigl((\omega_s\owedge
\omega_s)(X,Y,Z,V) +4\omega_s(X,Y)\omega_s(Z,V)\Bigr) \Big].
\end{multline}
The Koszul formula expressing the Levi-Civita connection in terms of the metric in the case of left-invariant vector fields $A,B,C$ on a Lie group reads
\begin{equation}  \label{lilc}
g(\nabla^g_AB,C)=\frac12\Big[
g([A,B],C))-g([B,C],A)+g([C,A],B)\Big].
\end{equation}
By Theorem~\ref{biqcon} we have the formula
\begin{equation}  \label{torhor}
T(X,Y)=2\sum_{s=1}^3\omega_s(X,Y)\xi_s.
\end{equation}
Using \eqref{torhor}, \eqref{lilc}, \eqref{lcbi} and the structure equations
\eqref{ex11} we found  the non zero coefficients of the curvature tensor are
$R(e_a,e_b,e_a,e_b)=-R(e_a,e_b,e_b,e_a)=1$, $a,b=1,\dots,4$,
$a\not=b$. Now \eqref{qcwdef2} yields $\qquad
W^{qc}(e_a,e_b,e_c,e_d)=R(e_a,e_b,e_c,e_d)=0$, when there are three
different indices in $a,b,c,d$. For the indices repeated in pairs we
have
\begin{multline*}
W^{qc}(e_a,e_b,e_a,e_b)=R(e_a,e_b,e_a,e_b)-\frac{1}8(g\owedge
g)(e_a,e_b,e_a,e_b)- \\
\frac{1}8\Big[\sum_{s=1}^3\Bigl((\omega_s\owedge\omega_s)(e_a,e_b,e_a,e_b)
+4\omega_s(e_a,e_b)\omega_s(e_a,e_b)\Bigr)
\Big]=1-\frac28-\frac68=0
\end{multline*}
Then Theorem \ref{main1} completes the proof.
\end{proof}

The  Lie algebra of the group $L_1$ is a  semi-direct sum, $l_1=su(2)\oplus_\pi \mathfrak{a}_{4,5}$, of $su(2)$ and the four dimensional solvable Lie algebra $\mathfrak{a}_{4,5}$, \cite{NB}, given, respectively, by
\begin{equation*}
\begin{aligned}
& su(2): \qquad d f^5=-\frac 12  f^{67}, \ d f^6=-\frac 12  f^{75},\ d f^7=-\frac 12  f^{56},\\
& \mathfrak{a}_{4,5}: \qquad d\tilde e^1 =0,\quad d\tilde e^2 =-\tilde e^{12},\quad d\tilde e^3 =-\tilde e^{13}, \quad d\tilde e^4= -\tilde e^{14}.
\end{aligned}
\end{equation*}
The action $\pi$ of $su(2)$ on $a_{4,5}$ is the restriction to  $su(2)$ of $ad$ on $L_1$, i.e., with the notation $\pi(f_i)\tilde e_j\equiv[f_i, \tilde e_j]$ the action is
\begin{equation*}
\begin{aligned}
& [f_5, \tilde e_3]=\frac 12 \tilde e_4, \ [f_5, \tilde e_4]=-\frac 12 \tilde e_3\\
& [f_6, \tilde e_4]=\frac 12 \tilde e_2, \ [f_6, \tilde e_2]=-\frac 12 \tilde e_4\\
& [f_7, \tilde e_2]=\frac 12 \tilde e_3, \ [f_7, \tilde e_3]=-\frac 12 \tilde e_2,
\end{aligned}
\end{equation*}
where $f_i$ and $\tilde e_j$ are the dual vectors.
This decomposition can be seen easier in the basis
\begin{equation*}
 f^1=e^1, \ f^2 =e^2, \ f^3=e^3, \ f^4=e^4,\ f^5=2e^2+e^5, \ f^6=2e^3+e^6,\ f^7=2e^4+e^7,
\end{equation*}
which satisfy the structure equations
\begin{equation*}
\begin{aligned}
& df^1=0, \ df^2=-f^{12}-\frac 12 f^{37}+\frac 12 f^{46},\ df^3=-f^{13} + \frac 12 f^{27} - \frac 12 f^{45}, \ df^4=-f^{14} - \frac 12 f^{26} + \frac 12 f^{35}\\
& d f^5=-\frac 12  f^{67}, \ d f^6=-\frac 12  f^{75},\ d f^7=-\frac 12  f^{56}.
\end{aligned}
\end{equation*}

\subsection{Example 2: zero torsion qc-non-flat structure.}

Consider the simply connected connected Lie group $L_2$ with Lie algebra
defined by the equations:
\begin{equation}  \label{ex2}
\begin{aligned}
& de^1 = 0, \quad de^2 = -e^{12} + e^{34},\quad de^3 = -\frac12
e^{13},\quad de^4 =-\frac12 e^{14},\\
& de^5 = 2 e^{12} + 2 e^{34} + e^{37} - e^{46} +\frac14e^{67}, \quad
de^6 = 2 e^{13} - 2 e^{24} -\frac12 e^{27} + e^{45}
-\frac14 e^{57}, \\
& de^7 = 2 e^{14} + 2 e^{23} + \frac12 e^{26} - e^{35} +\frac14
e^{56}.
\end{aligned}
\end{equation}
A global qc structure on $L_2$ is defined by \eqref{qc1}.
It is easy to check from \eqref{ex2} that the triple $\{\xi_1=e_5,
\xi_2=e_6, \xi_3=e_7\}$ form the Reeb vector fields satisfying
\eqref{bi1} and therefore the Biquard connection do exists.
\begin{thrm}
\label{m22} Let $(L_2,\eta,\mathbb{Q})$ be the simply connected Lie
group
with Lie algebra $\mathfrak {l_2}$  equipped with the left invariant qc
structure $(\eta,\mathbb{Q})$ defined above. Then:
\begin{itemize}
\item[a)] The qc structure is qc Einstein 
and the normalized qc scalar curvature is a negative
constant, $S=-\frac14$. \item[b)] The qc conformal curvature
$W^{qc}\not=0$ and  therefore $(L_2,\eta,\mathbb{Q})$ is not
locally qc conformally flat.
\end{itemize}
\end{thrm}

\begin{proof}
We compute the $sp(1)$-connection 1-forms and the horizontal Ricci forms of
the Biquard connection. The Lie algebra structure equations
\eqref{ex2} together with \eqref{coneforms}, \eqref{coneform1} and
\eqref{sp1curv} imply
\begin{equation}  \label{ex2conf}
\begin{aligned}& \alpha_1=-\frac12e^2-(\frac18+\frac{S}2)\eta_1, \quad
\alpha_2=-e^3-(\frac18+\frac{S}2)\eta_2,\quad
\alpha_3=-e^4-(\frac18+\frac{S}2)\eta_3,\\
&\rho_s(X,Y)=(\frac18-\frac{S}2)\omega_s(X,Y). \end{aligned}
\end{equation}
Compare \eqref{ex2conf} with \eqref{sixtyfour} to conclude that the
torsion is zero and the  normalized qc scalar $S=-\frac14$.
Theorem~\ref{sixtyseven} completes the proof of part a).

In view of Theorem~\ref{main1}, we have to show
$W^{qc}(e_1,e_2,e_3,e_4)=R(e_1,e_2,e_3,e_4)\not=0$.
Indeed, using \eqref{torhor},  \eqref{lilc}, \eqref{lcbi} and the
structure equations \eqref{ex2}
we found $R(e_1,e_2,e_3,e_4)=-\frac 12\not=0$.
\end{proof}

The  Lie algebra $l_2$ of the group  $L_2$ is a a direct sum of $su(2)$ and the four dimensional solvable algebra $\mathfrak{a}_{4,8}$,   \cite{NB}, given respectively by
\begin{equation*}
\begin{aligned}
& su(2): \qquad d\tilde e^5=\frac 14 \tilde e^{67}, \ d\tilde e^6=\frac 14 \tilde e^{75},\ d\tilde e^7=\frac 14 \tilde e^{56}\\
& \mathfrak{a}_{4,8}: \qquad d\tilde e^1 =0,\quad d\tilde e^2 =-\tilde e^{12} + \tilde e^{34},\quad d\tilde e^3 =-\frac 12\tilde e^{13}, \quad d\tilde e^4= -\frac 12\tilde e^{14}.
\end{aligned}
\end{equation*}
This decomposition can be seen by letting
\begin{equation*}
 e^5=-2\tilde e^2 + \tilde e^5, \ e^6=-4\tilde e^3 + \tilde e^6,\ e^7=-4\tilde e^4 + \tilde e^7, \quad e^m =\tilde e^m, \ m=1,2,3,4.
\end{equation*}

\subsection{Example 3: non-zero torsion qc-non-flat structure.}

Consider the {solvable indecomposable}  Lie algebra $\mathfrak {l_3}$  defined by the equations
\begin{equation}  \label{ex31}
\begin{aligned}
& de^1=-\frac32e^{13}+\frac32e^{24}-\frac34e^{25}+\frac14e^{36}-\frac14e^{47}+
\frac18e^{57},\\
& de^2=-\frac32e^{14}-\frac32e^{23}+\frac34e^{15}+\frac14e^{37}+\frac14e^{46}-
\frac18e^{56},\\
& de^3=0, \qquad
de^4=e^{12}+e^{34}+\frac12e^{17}-\frac12e^{26}+\frac14e^{67}, \\
& de^5=2e^{12}+2e^{34}+e^{17}-e^{26}+\frac12e^{67},\\ &
de^6=2e^{13}+2e^{42}+e^{25},\qquad de^7=2e^{14}+2e^{23}-e^{15},
\end{aligned}
\end{equation}
and let $e_l, 1 \leq l\leq 7$ be the left invariant vector field dual
to the 1-forms ${e^l, 1 \leq l\leq 7}$. We define a
global qc structure on the corresponding Lie group $L_3$ by \eqref{qc1}.
It follows from \eqref{ex31} that the triple $\{\xi_1=e_5,
\xi_2=e_6, \xi_3=e_7\}$ form the Reeb vector fields satisfying
\eqref{bi1} and therefore the Biquard connection exists.

\begin{thrm}
\label{m3} Let $(L_3,\eta,\mathbb{Q})$ be the simply connected connected Lie
group
with Lie algebra $\mathfrak {l_3}$  equipped with the left invariant qc
structure $(\eta,\mathbb{Q})$ defined by \eqref{qc1}. Then
\begin{itemize}
\item[a)] The torsion endomorphism of the Biquard connection is
not zero and the normalized qc scalar
curvature is negative, $S=-1$.
\item[b)] The qc conformal curvature  $W^{qc}\not=0$, and
therefore $(L_3,\eta,\mathbb{Q})$ is not locally qc conformally
flat.
\end{itemize}
\end{thrm}
\begin{proof}
The  structure equations
\eqref{ex31} together with \eqref{coneforms}, \eqref{coneform1}
imply
\begin{equation}  \label{ex3conf}
\alpha_1=(\frac14-\frac{S}2)\eta_1, \qquad \alpha_2=-e^1-(\frac14+\frac{S}
2)\eta_2,\qquad \alpha_3=-e^2-(\frac14+\frac{S}2)\eta_3.
\end{equation}
Now, \eqref{ex3conf}, \eqref{ex31} and \eqref{sp1curv} yield
\begin{equation}  \label{rtor}
\begin{aligned}
\rho_1(X,Y)=\frac12\Big[(\frac12-S)(e^{12}+e^{34})+e^{12}\Big](X,Y)=
\frac14(e^{12}-e^{34})(X,Y)+ \frac12(1-S)\omega_1(X,Y),\\
\rho_2(X,Y)=\frac12\Big[\frac32(e^{13}-e^{24})(X,Y)-(
\frac12+S)(e^{13}-e^{24})(X,Y)\Big]= +\frac12(1-S)\omega_2(X,Y),\\
\rho_3(X,Y)=\frac12\Big[\frac32(e^{14}+e^{23})(X,Y)-(
\frac12+S)(e^{14}+e^{23})(X,Y)\Big]= +\frac12(1-S)\omega_3(X,Y).
\end{aligned}
\end{equation}
Comparing \eqref{rtor} with \eqref{sixtyfour} we conclude
\begin{equation}  \label{tr3}
\begin{aligned} T^0(X,I_1Y)-T^0(I_1X,Y)=\frac12(e^{12}-e^{34})(X,Y), \qquad
S=-1,\\ T^0(X,I_2Y)-T^0(I_2X,Y)=0, \quad T^0(X,I_3Y)-T^0(I_3X,Y)=0.
\end{aligned}
\end{equation}
Now  we compute $W^{qc}$.
Denote $\psi=-\frac14(e^{12}-e^{34})$ and compare \eqref{tr3}
with \eqref{propt} and \eqref{need} to get
\begin{equation}  \label{tor3}
T^0(X,Y)=\psi(X,I_1Y), \qquad
g(T(\xi_s,X),Y)=-\frac14(\psi(I_sX,I_1Y)+\psi(X,I_1I_sY)).
\end{equation}
Using $U=0$ and  \eqref{propt} we conclude
from \eqref{qcwdef1} that $ W^{qc}(e_1,e_2,e_3,e_4)=R(e_1,e_2,e_3,e_4)$ since
other terms on the right hand
side of \eqref{qcwdef1} vanish on the
quadruple $\{e_1,e_2=-I_1e_1,e_3=-I_2e_1,e_4=-I_3e_1\}$.
Using \eqref{lcbi}, \eqref{lilc},
\eqref{torhor}, \eqref{ex31} and \eqref{tor3}, we obtain $R(e_1,e_2,e_3,e_4)=-\frac12\not=0$.
\end{proof}

\section{$Sp(n)Sp(1)$-hypo structures and  quaternionic K\"ahler manifolds}
\label{qsph}

Guided by Examples 0--3 in the last Section we relax the definition of a qc
structure by removing the ``contact condition''
$d\eta_{s_{|H}}=2\omega_s$. In this way we
come to an $Sp(n)Sp(1)$ structure (almost $3$-contact structure
see \cite{Kuo}). The goal  is to obtain a geometric structure which may induce a
quaternionic K\"ahler metric on a product of the given manifold  with (an interval of)  the real line.

\begin{dfn}\label{d:sp(n)sp(1)structure}
An $Sp(n)Sp(1)$ structure on a $(4n+3)$-dimensional Riemannian
manifold $(M,g)$ is a codimension three distribution $H$ { such
that}
\begin{enumerate}
\item[i)]  $H$ has an $Sp(n)Sp(1)$ structure, that is, it is
equipped with a Riemannian metric $g$ and a rank-three bundle
$\mathbb Q$ consisting of (1,1)-tensors on $H$ locally generated by three
almost complex structures $I_1,I_2,I_3$ on $H$ satisfying the identities
of the imaginary unit quaternions, $I_1I_2=-I_2I_1=I_3, \quad
I_1I_2I_3=-id_{|_H}$ which are hermitian compatible with
the metric $g(I_s.,I_s.)=g(.,.)$, i.e. $H$ has an almost
quaternion hermitian structure.
\item[ii)] $H$
is locally given as the kernel of a 1-form
$\eta=(\eta_1,\eta_2,\eta_3)$ with values in $\mathbb{R}^3$.
\end{enumerate}
The local fundamental 2-forms are defined on $H$ as usual by
$\omega_s(X,Y)=g(I_sX,Y)$.
\end{dfn}
If the first Pontrjagin class of $M$ vanishes then
the 1-forms $\eta_s$ as well as the fundamental 2-forms $\omega_s$
are globally defined \cite{AK}.

\begin{dfn} We define a global $Sp(n)Sp(1)$-invariant
4-form of an $Sp(n)Sp(1)$ structure $(M,g,\mathbb Q)$ on a
$(4n+3)$-dimensional manifold $M$ by the formula
\begin{equation}\label{res4}
\Omega_{\mathbb Q}=\omega_1^2+\omega_2^2+\omega_3^2+2\omega_1\wedge\eta_2\wedge\eta_3+
2\omega_2\wedge\eta_3\wedge\eta_1+2\omega_3\wedge\eta_1\wedge\eta_2.
\end{equation}
\end{dfn}

Let $M^{4n+4}$ be a $(4n+4)$- dimensional manifold equipped
with an $Sp(n+1)Sp(1)$ structure, i.e.\\ $(M^{4n+4},g,J_1,J_2,J_3)$ is
an almost quaternion hermitian manifold with local K\"ahler forms
$F_i(.,.)=g(J_i.,.)$. The  fundamental 4-form
\begin{equation}\label{4-form}\Phi=F_1\wedge F_1+F_2\wedge
F_2+F_3\wedge F_3
\end{equation}
is globally defined and encodes fundamental properties of the
structure. If the holonomy of the Levi-Civita connection is
contained in $Sp(n+1)Sp(1)$  then the manifold is a quaternionic
K\"ahler  manifold which is consequently an Einstein manifold.
Equivalent conditions are either that  the fundamental 4-form $\Phi$ is parallel with respect
to the Levi-Civita connection or $\Phi$ is a closed form and
\begin{equation}\label{qkcon}
dF_s=0 \quad {\text mod}\quad  \{F_i,F_j,F_k\}
\end{equation}
\cite{Sw2}. The latter is equivalent to the fact
that the fundamental 4-form is closed ($d\Phi=0$) provided the
dimension is strictly bigger than eight (\cite{Sw,Sw2, Sal}) with a
counter-example in dimension eight constructed by Salamon in
\cite{Sal}.

Let $f: N^{4n+3}\longrightarrow M^{4n+4}$ be an oriented
hypersurface of $M^{4n+4}$ and denote by $\mathbb{N}$ the unit
normal vector field. Then an $Sp(n+1)Sp(1)$ structure on $M$ induces
an $Sp(n)Sp(1)$ structure on $N^{4n+3}$ locally given by
$(\eta_s,\omega_s)$  defined by the equalities
$
\eta_s=\mathbb{N}\lrcorner F_s,\quad
\omega_i=f^*F_i-\eta_j\wedge\eta_k$.
The
fundamental four form $\Phi$ on $M$ restricts to the
four form $\Omega_{\mathbb Q}$ on $N$,
$
\Omega_{\mathbb Q}=f^*\Phi=(f^*F_1)^2+(f^*F_2)^2+(f^*F_3)^2.
$

Suppose that $(M^{4n+4},g)$ has holonomy contained in
$Sp(n+1)Sp(1)$. Then $d\Phi=0$
implies that the $Sp(n)Sp(1)$ structure  induced on $N^{4n+3}$
satisfies the equation
\begin{equation}  \label{qrhypo}
d\Omega_{\mathbb Q}=0,
\end{equation}
since $d$ comutes with $f^*$, $df^*=f^*d$.
\begin{dfn}
An $Sp(n)Sp(1)$ structure $(M,g,\mathbb Q)$ on a
$(4n+3)$-dimensional manifold is called \textrm{Sp(n)Sp(1) - hypo}
if the 4-form $\Omega_{\mathbb Q}$ is closed, $d\Omega_{\mathbb Q}=0$.
\end{dfn}

Hence, any oriented hypersurface $N^{4n+3}$ of a quaternionic
K\"ahler $M^{4n+4}$ is naturally endowed with an $Sp(n)Sp(1)$-hypo
structure.

Vice versa, a $(4n+3)$-manifold $N^{4n+3}$ with an $Sp(n)Sp(1)$
structure $(\eta_s,\omega_s)$ induces an $Sp(n+1)Sp(1)$ structure
$(F_s)$ on $N^{4n+3}\times \mathbb{R}$ defined by
\begin{equation}  \label{qhyp0}
F_i=\omega_i+\eta_j\wedge\eta_k-\eta_i\wedge dt,
\end{equation}
where $t$ is a coordinate on $\mathbb{R}$.

Consider the family of $Sp(n)Sp(1)$ structures $(\eta_s(t),\omega_s(t))$ on
$N^{4n+3}$
depending on a real parameter $t\in\mathbb{R}$, and the
corresponding $Sp(n+1)Sp(1)$ structures $F_s(t)$ on $N^{4n+3}\times\mathbb{R}$.

\begin{prop}\label{qevpro} An $Sp(n)Sp(1)$
structure $(\eta_s,\omega_s)$ on $N^{4n+3}$ can be lifted to an almost
quaternionic hermitian  structure $(F_s(t))$ with a closed four form on $N^{4n+3}\times\mathbb{R}$
defined by \eqref{qhyp0} if and only if it is an $Sp(n)Sp(1)$-hypo
structure
which generates a 1-parameter family of $Sp(n)Sp(1)$-hypo
structures $(\eta_s(t),\omega_s(t))$ satisfying the following \emph{evolution
Sp(n)Sp(1)-hypo equations}
\begin{equation}  \label{qevolunk}
\partial_t\Omega_{\mathbb Q}(t)=d\Big[6\eta_1(t)\wedge\eta_2(t)\wedge\eta_3(t)+2\omega_1(t)\wedge\eta_1(t)+
2\omega_2(t)\wedge\eta_2(t)+2\omega_3(t)\wedge\eta_3(t)\Big],
\end{equation}
where $d$ is the exterior derivative on $N$.

If $n\geq2$, then the almost
 quaternionic hermitian  structure $(F_s(t))$ with a closed four form on $N^{4n+3}\times\mathbb{R}$
defined by \eqref{qhyp0} is quaternionic K\"ahler.
\end{prop}

\begin{proof}
If we apply \eqref{qhyp0} to \eqref{4-form} and then take  the
exterior derivative in the obtained equation we see that the
equality $d\Phi=0$ holds precisely when \eqref{qrhypo} and
\eqref{qevolunk} are fulfilled.

It remains to show that the equations \eqref{qevolunk} imply
that \eqref{qrhypo} holds for each $t$. Using \eqref{qevolunk}, we
calculate
\begin{gather*}
\partial_td\Omega_{\mathbb Q}= d^2\Big[6\eta_1(t)\wedge\eta_2(t)\wedge\eta_3(t)+2\omega_1(t)\wedge\eta_1(t)+
2\omega_2(t)\wedge\eta_2(t)+2\omega_3(t)\wedge\eta_3(t)\Big]=0.
\end{gather*}
Hence, the equalities \eqref{qrhypo} are independent of $t$ and
therefore valid for all $t$ since they hold  for
$t=0$.
\end{proof}
It is  interesting to know whether the converse of Proposition~\ref{qevpro} holds, i.e., is
it true that any $Sp(n)Sp(1)$-hypo structure
on $N^{4n+3}, \ n>2$, can be lifted to a
quaternionic K\"ahler structure on $N^{4n+3}\times\mathbb{R}$?

Solutions to \eqref{qrhypo} are given in the case of 3-Sasakian
manifolds in \cite{WM}. We construct explicit
examples relying on the properties of the qc structures for which we solve the evolution equations \eqref{qevolunk}.

\subsection{Quaternionic contact and $Sp(n)Sp(1)$-hypo structures}
We show the conditions under which a qc structure
is an $Sp(n)Sp(1)$-hypo. For $n=1$, we prove that it happens exactly when
the vertical distribution is integrable, while for $n>1$ we have that a qc structure is $Sp(n)Sp(1)$-hypo if and only if it is qc Einstein.
To this end we recall that the structure equations of a qc structure, discovered in \cite{IV1},
are
\begin{gather}  \label{streq}
2\omega_i=d\eta_i+\eta_j\wedge\alpha_k-\eta_k\wedge\alpha_j + S
\eta_j\wedge\eta_k,
\end{gather}
and a 3-Sasakian qc structure is characterized by the structure equations
$
2\omega_i=d\eta_i-2 \eta_j\wedge\eta_k.
$

The $Sp(n)Sp(1)$-invariant horizontal 4-form, called the fundamental 4-form,  is defined in \cite{IV1} by
 \begin{equation}\label{fund4}
\Omega=\omega_1\wedge\omega_1+\omega_2\wedge\omega_2+\omega_3\wedge\omega_3.
\end{equation}
If the dimension of the qc manifold  is bigger than seven it turns out that qc Einstein condition 
is equivalent the
fundamental $4$-form $\Omega$
being closed, see \cite{IV1}.
\begin{prop}\label{qchypo}
A qc structure is $Sp(n)Sp(1)$-hypo structure if and only if its fundamental four form is closed, $d\Omega=0$.
\end{prop}
\begin{proof}
Comparing the definitions of $\Omega_{\mathbb Q}$ and $\Omega$ we see that it is sufficient to show that the four form
$$\omega_1\wedge\eta_2\wedge\eta_3+\omega_2\wedge\eta_3\wedge\eta_1+\omega_3\wedge\eta_1\wedge\eta_2$$
is closed. The structure equations \eqref{streq} imply \cite{IV1}
\begin{equation}\label{strom}
d\omega_i=\omega_j\wedge(\alpha_k+S\eta_k)-\omega_k\wedge(\alpha_j+S\eta_j)-\rho_k\wedge\eta_j+
\rho_j\wedge\eta_k+\frac12dS\wedge\eta_j\wedge\eta_k.
\end{equation}
Using \eqref{streq} and \eqref{strom} we obtain the validity of the next
\begin{lemma}\label{closedqc}
For any qc structure we have
$$d(\omega_1\wedge\eta_2\wedge\eta_3+\omega_2\wedge\eta_3\wedge\eta_1+\omega_3\wedge\eta_1\wedge\eta_2)=0.$$
\end{lemma}
Now, Lemma~\eqref{closedqc} implies $d\Omega_{\mathbb Q}=0$ precisely when $d\Omega=0$.
\end{proof}
The main theorem of \cite{IV1} asserts that a qc strucutre on a qc manifold of dimension strictly bigger than seven has closed fundamental four form exactly when it has zero torsion endomorphism of the Biquard connection. In dimension seven we have the following result.
\begin{thrm}\label{integr-closed}
For a qc structure in dimension seven the next conditions are equivalent.
\begin{itemize}
\item [a)] The fundamental four form is closed, $d\Omega=0$.
\item [b)] The vertical distribution is integrable.
\item [c)] The qc structure is $Sp(1)Sp(1)$-hypo structure.
\end{itemize}
\end{thrm}
\begin{proof}
In the case of dimension seven we have the identities $\omega_s\wedge\omega_p=\delta_{sq}vol_H$. Therefore, $\Omega=3\omega_1\wedge\omega_1$. Applying \eqref{strom} we get
\begin{equation}\label{7777}
d\Omega=6d\omega_1\wedge\omega_1=6\omega_1\wedge\rho_2\wedge\eta_3-6\omega_1\wedge\rho_3\wedge\eta_2+3dS\wedge\omega_1\wedge\eta_2\wedge\eta_3.
\end{equation}
For a 1-form $\beta$ we use the convention $I_s\beta(X)=-\beta(I_sX)$. We obtain from \eqref{7777} that the nonzero parts of $d\Omega$ are given by
 \begin{gather}\label{771}
 d\Omega(X,Y,Z,\xi_1,\xi_2)=6\omega_1\wedge(\xi_1\lrcorner\rho_3)(X,Y,Z);\quad d\Omega(X,Y,Z,\xi_1,\xi_3)=6\omega_1\wedge(\xi_1\lrcorner\rho_2)(X,Y,Z);\\\label{772}
 d\Omega(X,Y,Z,\xi_2,\xi_3)=-3\omega_1\wedge\Big[2(\xi_2\lrcorner\rho_2)+2(\xi_3\lrcorner\rho_3)-dS\Big](X,Y,Z)=6\omega_1\wedge I_3(\xi_2\lrcorner\rho_1)(X,Y,Z);\\\label{773}
 d\Omega(X,Y,\xi_1,\xi_2,\xi_3)=3\omega_1(X,Y)\Big[2\rho_2(\xi_1,\xi_2)+2\rho_3(\xi_1,\xi_3)+dS(\xi_1)\Big]=0;\\\label{774}
 d\Omega(X,Y,Z,U,\xi_3)=6\omega_1\wedge\rho_2(X,Y,Z,U)=g(\omega_1,\rho_2)vol_H=0,
 \end{gather}
 where we apply the last equation in \eqref{sixtyfour} to obtain the last equality in \eqref{772}.
 The second equality in \eqref{773} is precisely the second formula of the third line of \eqref{sixtyfour}. The last equality in \eqref{774} follows from the first formula in \eqref{sixtyfour} which says that the horizontal two form $\rho_2$ is of type (1,1) with respect to $I_2$ and therefore it is orthogonal to the 2-form $\omega_1$ which is of type (2,0)+(0,2) with respect to $I_2$.

  Assume that the fundamental four form is closed. Then \eqref{771} and \eqref{772} imply that the 1-forms $(\xi_s\lrcorner\rho_t)|_H$ vanish for $s\not= t$ which is equivalent the vertical distribution to be integrable due to the third line of \eqref{sixtyfour}.

 For the converse, let the vertical distribution be integrable, i.e. $(i_{\xi_t}\rho_s)|_H=0$ for $s\not=t$. Then the last formula in \eqref{sixtyfour} yields $(i_{\xi_s}\rho_s)|_H=\frac14dS|_H$. Substitute the latter into \eqref{772} to obtain the vanishing of this term which combined with \eqref{771}, \eqref{773} and \eqref{774}  yields   $d\Omega=0$. This  proves $a)\Leftrightarrow b)$. Finally, Proposition~\ref{qchypo} completes the proof.
\end{proof}

 Combining Proposition~\ref{qchypo} with  the main theorem of \cite{IV1} we obtain the following Corollary.
\begin{cor}\label{qchypo1}
a) A qc structure on a (4n+3)-dimensional manifold is $Sp(n)Sp(1)$-hypo if and only if it is qc Einstein
provided $n>1$.

b) In dimension seven,  a qc Einstein structure  
with constant qc scalar curvature is $Sp(1)Sp(1)$-hypo structure.

In  both cases we have the structure equations
\begin{equation}\label{strom1}
d\omega_i=\omega_j\wedge\alpha_k-\omega_k\wedge\alpha_j.
\end{equation}
\end{cor}
\begin{proof}
If $n>1$ and $T^0=U=0$, then  Theorem~\ref{sixtyseven} implies
\begin{equation}\label{closed}
\rho_l(X,Y)=-S\omega_l(X,Y),  \quad \rho_l(\xi_m,X)=0,\quad
\rho_i(\xi_i,\xi_j)+\rho_k(\xi_k,\xi_j)=0,
\end{equation}
since $S$ is constant  and the vertical distribution is integrable \cite{IMV}. In dimension seven the second equality of \eqref{closed} is a consequence of Theorem~3.1 in \cite{IV} which expresses  $\rho_l(\xi_m,X)$ in terms of the covariant derivatives of the torsion and the differential of the qc scalar curvature.  Applying \eqref{closed} into \eqref{strom} we get \eqref{strom1}.
\end{proof}

\subsection{Construction of quaternionic K\"ahler structures using qc structures}
In this section we construct explicit quaternionic K\"ahler metrics on the product of a qc manifold with a real interval.
\begin{thrm}\label{qkmetric}
Let $(M,g,\mathbb Q)$ be a smooth qc Einstein manifold of dimension $4n+3$ 
and, in dimension seven, with  constant normalized qc scalar curvature $S$.
 For a suitable constant $a$, the manifold $M\times\mathbb{R}$ has a
quaternionic K\"ahler structure given by the following metric and
fundamental 4-form
\begin{equation}\label{e:general qK metric}
\begin{aligned}
&g=ug_H+(\frac12S u+au^2)(\eta_1^2+\eta_2^2+\eta_3^2)+\frac {1}{2(S u+2au^2)}(du)^2,\quad S u+2au^2>0,\\
&\Phi=F_1\wedge F_1 +F_2\wedge F_2 + F_3\wedge F_3,
\end{aligned}
\end{equation}
where locally
\begin{equation}\label{e:local qK forms}
\begin{aligned}
F_i(u)=u\omega_i+(a u^2+\frac12S u)\,\eta_j\wedge\eta_k-\frac 12
\eta_i\wedge du.
\end{aligned}
\end{equation}
The Ricci tensor of the quaternionic K\"ahler metric is  $Ric=-4(n+3)a g$.
\end{thrm}
\begin{proof}
Let $h$ and $f$ be some functions of the unknown $t$ . Consider the 2-forms defined by
\begin{equation}\label{fff}F_i(t)=f(t)\omega_i+h^2(t)\eta_j\wedge \eta_k-h(t)\eta_i\wedge dt.
\end{equation}
Let $\Phi$ be given with the second equation  in \eqref{e:general qK metric}. A direct
calculation applying \eqref{streq} and \eqref{strom1} shows that ($\Sigma_{(ijk)}$ means the cyclic sum)
\begin{equation*}
d\Phi=\Sigma_{(ijk)}\left [ \left (
(f^2)'-4fh\right)\omega_i\wedge\omega_i\wedge dt+\left ( 2\left
(fh^2 \right )' +2S fh-12h^3 \right
)\omega_i\wedge\eta_j\wedge\eta_k\wedge dt \right ].
\end{equation*}
Thus, if we take $h=\frac 12 f'$ we come to
$$d\Phi=f'\Sigma_{(ijk)}(-f'^2+f f''+S
f)\omega_i\wedge\eta_j\wedge\eta_k\wedge dt,$$
 which shows that
$\Phi$ is closed exactly when
\begin{equation}\label{solqk7}
ff''-f'^2+S f=0, \qquad h=\frac 12 f'.
\end{equation}
With the help of the substitution $v=-\ln f$ we see that $\left (
\frac {dv}{dt}\right )^2=2S e^v +4a$ for any constant $a$. This
shows that $\left(  \frac {dt}{df}  \right)^2 =\left( \frac {dt}{dv}
\right )^2 \left( \frac {dv}{df}\right )^2=\frac {1}{2(S f +
2af^2)} > 0$ and $h^2=\frac12S f + af^2$. Renaming $f$ to $u$ gives the
quaternionic structure in the local form \eqref{e:local qK forms}
and the metric in \eqref{e:general qK metric}.

In dimension seven, in order to see that
$<F_1, F_2, F_3>$ is a differential ideal when equations \eqref{solqk7} hold,  we need to compute the
differentials $dF_i$. Using \eqref{streq} and \eqref{strom1} we obtain taking the exterior derivative of \eqref{fff} that
\begin{multline}\label{clideal}
dF_i=\Big(\alpha_k+\frac{2h^2}f\eta_k\Big)\wedge F_j-\Big(\alpha_j+\frac{2h^2}f\eta_j\Big)\wedge F_k\\+(f'-2h)\omega_i\wedge dt +(2hh'+hS-\frac{4h^3}4)\eta_j\wedge\eta_k\wedge dt.
\end{multline}
An easy application of \eqref{solqk7} anihilates the second line of \eqref{clideal} which proves that the defined structure is
quaternionic K\"ahler due to  Swann's theorem \cite{Sw} recalling that we have also \eqref{solqk7}, i.e., $\Phi$ is a closed form.
\end{proof}

Using the above theorem we obtain explicit  quaternionic K\"ahler structures based on examples of qc structures with vanishing torsion. We turn to their description in the following subsection.

\subsection{Quatenionic K\"ahler metrics based on qc  Einstein structure with zero qc scalar curvature}

When the qc scalar curvature vanishes, $S=0$, we let $a=b^2$, $ u=e^{2b\sigma}$ in the above Theorem to obtain the next Corollary.
\begin{cor}\label{qkmetricflat}
Let $(M,g,\mathbb Q)$ be a smooth qc Einstein manifold of dimension $4n+3$ with vanishing  qc scalar curvature.  For any non-zero constant $b$, the manifold $M\times\mathbb{R}$ has a quaternionic K\"ahler structure given by the following metric and
fundamental 4-form
\begin{equation}\label{e:general qK metric-f}
\begin{aligned}
&g=e^{2b\sigma}g_H+b^2e^{4b\sigma}(\eta_1^2+\eta_2^2+\eta_3^2)+(d\sigma)^2,\\
&\Phi=F_1\wedge F_1 +F_2\wedge F_2 + F_3\wedge F_3,
\end{aligned}
\end{equation}
where locally
\begin{equation}\label{e:local qK forms-f}
\begin{aligned}
F_i=e^{2b\sigma}\omega_i+b^2e^{4b\sigma}\,\eta_j\wedge\eta_k-be^{2b\sigma}
\eta_i\wedge d\sigma.
\end{aligned}
\end{equation}
In particular, the quaternionic K\"ahler metric on $M\times\mathbb{R}$ is complete if  the metric on $M$ is complete.
\end{cor}
\begin{proof}
The completeness follows similarly to the case of a warped product with strictly positive warping function, see \cite{ONeill}
\end{proof}

\subsubsection{Quaternionic K\"ahler metrics from the quaternionic
Heisenberg group} Consider the $(4n+3)$ - dimensional
quaternionic Heisenberg group $\mathbb{G}^n$, viewed as a qc
structure. According to \eqref{e:general qK metric-f}, the metric
\begin{equation}\label{e:aqKahler}
g=e^{2b\sigma}\left ((e^1)^2+\dots +(e^{4n})^2 \right )+b^2e^{4b\sigma}\left ( (\eta_1)^2+(\eta_2)^2+(\eta_3)^2
 \right )+d\sigma^2
\end{equation}
is a complete quaternionic K\"ahler metric in dimensions $4n+4$ with
$n\geq 1$.
 The Einstein constant is negative and equal to
$ -4(n+3)b^2$.
 This complete Einstein
 metric has  been found in dimension  eight on a solvable Lie group as an Einstein metric starting with  a $T^3$ bundle
 over $T^4$ in \cite[equation (148)]{GLPS}. Thus, the Einstein metric in dimension eight discovered in \cite{GLPS} is
 in fact a quaternionic K\"ahler metric. Similarly to the eight dimensional case, the metric \eqref{e:aqKahler}
 is a left invariant metric on a $4n+4$ dimensional Lie group. In order to see this we use the structure equation
 \eqref{4n+3heis} giving $de^i=0$, hence the one forms $\tilde e^i =e^{b\sigma}e^i$, $\tilde \eta_i=e^{2b\sigma}\eta_i$ and $d\sigma$
 define a $4n+4$-dimensional Lie algebra
 \begin{equation*}
  d\tilde e^i=-b\tilde e^i\wedge d\sigma, \qquad d\tilde \eta_i=2\tilde \omega_i -2b\tilde\eta_i\wedge d\sigma,
 \end{equation*}
where $\tilde \omega_i$ is obtained from $\omega_i$ by replacing $e^i$ with $\tilde e^i$.

 The explicit description of the quaternionic K\"ahler metric can be obtained from the qc structure of the quaternionic Heisenberg group (using the form of \cite{IMV}),
\begin{equation}
\begin{aligned}\label{e:Heisenbegr ctct forms}
\eta_1\ =\ \frac 12\ dx\ -\ x^\alpha d t^\alpha\  +\ t^\alpha d
x^\alpha\ -\ z^\alpha d y^\alpha \ +\ y^\alpha d z^\alpha,\\
 \eta_2\ =\ \frac 12\ dy\ -\ y^\alpha d t^\alpha\
+\ z^\alpha d x^\alpha\ + \ t^\alpha d y^\alpha \ -\ x^\alpha d z^\alpha,\\
\eta_3\ =\ \frac 12\ dz\ -\ z^\alpha d t^\alpha\  - \ y^\alpha d x^\alpha\ + \
x^\alpha d y^\alpha \ +\ t^\alpha d z^\alpha,
\end{aligned}
\end{equation}
with summation over $\alpha=1,\dots,n$. The horizontal forms $e^i$ are $d t^\alpha$, $dx^\alpha$, $dy ^\alpha$ and $dz ^\alpha$.

\subsection{Quatenionic K\"ahler metrics based on a qc Einstein structure with negative qc scalar curvature}

\vskip.1truein

Let us consider a qc Einstein structure with a negative qc scalar
curvature. Accordingly, we let $S=-k^2<0$ and also replace $a$ in
Theorem \ref{qkmetric} with $\frac 12 b^2k^2>0$. With this
notation and in terms of the coordinate $\sigma$, $u=\frac {1}{2b^2}
(1+\cosh \sigma)$, we have $Su+2au^2=k^2(b^2u^2-u)=\frac {k^2}{4b^2}
\sinh^2 \sigma$ and the metric \eqref{e:general qK metric} takes the form
\begin{equation}\label{e:QK subR}
g=\frac {1}{2b^2} (1+\cosh \sigma)\ g_H+\frac {k^2}{8b^2} \sinh^2 \sigma\ (\eta_1^2+\eta_2^2+\eta_3^2)+\frac {1}{2k^2b^2}\,d\sigma^2,
\end{equation}
defined on $\hat M  =M\times\mathbb{ R}$.  The Ricci tensor of $g$ is then $Ric=-2(n+3) b^2 k^2 g$. Notice that the above
metric degenerates only when $\sigma=0$, but defines a sub-Riemannian
metric on the distribution $\hat H$ spanned by $H$ and $\frac
{\partial}{\partial \sigma}$ since $(1+\cosh \sigma)\geq 2$ and $\hat H$
generates the whole tangent space. From the formula for $g$ it is
apparent that if  $v$ is any tangent vector to $M\times\mathbb{
R}$ then $g(v,v)\geq |d\sigma(v)|^2$ and $g(v,v)\geq g_H(v,v)$.
Furthermore, if $\gamma$ is a horizontal curve on
$M\times\mathbb{ R}$, i.e., $\dot{\gamma}\in \hat H$, then its
projection on $M$ is also  horizontal curve.

Thus, the lengths of the projections on $\mathbb{R}$ and $M$ of
any horizontal curve on $M\times\mathbb{ R}$ are less than the
length of the horizontal curve on $\hat M$.  Let $(p_n,\sigma_n)\in
\hat M$ be a Cauchy sequence in the sub-Riemannian metric $g$.
From the argument so far we see  that the sequences $\sigma_n\in
\mathbb{R}$  and $p_n\in M$ are Cauchy sequences in the metric
$d\sigma^2$ and the sub-Riemannian metric  $g_H$ on $M$, respectively.
It follows that if the sub-Riemannian metric $g_H$ is complete
then the  sub-Riemannian metric  $g$ on $\hat M$ is complete. As
far as completeness of sub-Riemannian metrics is concerned, it is
useful to have in mind the result of \cite[Theorem 7.4]{Str}
according to which if a sub-Riemannian metric has a  Riemannian
contraction which is complete, then the sub-Riemannian metric is
also complete. In particular, if $g_H+ \eta_1^2+\eta_2^2+\eta_3^2$
is a complete metric on $M$, then $g_H$ defines a complete
sub-Riemannian metric on $M$ and $g$ defines a complete
sub-Riemannian metric on $M\times\mathbb{ R}$.

Next, we give some examples.

\subsubsection{Explicit quaternionic K\"ahler metrics from the zero-torsion qc-flat qc structure on $\mathfrak {l_1}$.}

As example of the above construction  we consider the Lie group $ {L_1}$ defined by the structure equations
\eqref{ex11}, which can be described  in local coordinates
$\{t,x,y,z,x_5,x_6,x_7\}$ as follows
\begin{equation}\label{loccoord}
\begin{aligned}
& e^1=-dt,\\
& e^2={\frac 12}\,{x_6}\,dx+{\frac 12}\,{x_5}\cos x\,dy
 +({\frac 12}\,{x_6}\cos y+{\frac 12}{x_5}\sin y\sin x)\,dz-{\frac 12}\,{x_7}\,dt+{\frac 12}\,d{x_7},\\
& e^3=-{\frac 12}\,{x_7}\,dx+{\frac 12}\,{x_5}\sin x\,dy
+(-{\frac 12}\,{x_7}\cos y-{\frac 12}{x_5}\sin y\,\cos x)\,dz-{\frac 12}\,{x_6}\,dt+{\frac 12}\,d{x_6},\\
& e^4=(-{\frac 12}{x_7}\cos x\,-{\frac 12}{x_6}\sin x\,)\,dy
-{\frac 12}\sin y\,(-{x_6}\cos x+{x_7}\sin x)\,dz-{\frac 12}\,{x_5}\,dt+{\frac 12}\,d{x_5},\\
& \eta_1=e^5=-{x_6}\,dx+(-{x_5}\cos x-2\sin x)\,dy\\
&\hskip1.7truein
+(-{x_6}\cos y-\sin y\sin x\,{x_5}+2\sin y\cos x)\,dz+{x_7}\,dt-d{x_7},\\
& \eta_2=e^6={x_7}\,dx+(2\cos x-{x_5}\sin x)\,dy\\ &\hskip1.7truein
 +({x_7}\cos y+2\sin y\sin x+{x_5}\sin y\,\cos x)\,dz+{x_6}\,dt-d{x_6},\\
& \eta_3=e^7=-2\,dx+(\cos x\,{x_7}+{x_6}\sin x)\,dy\\
&\hskip1.7truein +(-2\cos y+{x_7}\sin y\sin x-{x_6}\sin y\,\cos
x)\,dz+{x_5}\,dt-d{x_5}.
\end{aligned}
\end{equation}
In this case $S=-\frac12$ in \eqref{e:general qK metric}. According to
\eqref{e:QK subR}, the
corresponding quaternionic  K\"ahler metric is
\begin{equation}\label{qknew8}
g=\frac {1}{2b^2} (1+\cosh \sigma)\ g_H+\frac {1}{16b^2} \sinh^2 \sigma\
(\eta_1^2+\eta_2^2+\eta_3^2)+\frac {1}{b^2}\,d\sigma^2,
\end{equation}
The quaternionic K\"ahler two forms are
\begin{equation*}
F_i(\sigma)= \frac {1}{2b^2} (1+\cosh \sigma)\omega_i+\frac {1}{16b^2}
\sinh^2 \sigma\,\eta_j\wedge\eta_k-\frac 12 \eta_i\wedge d\sigma.
\end{equation*}
The Ricci tensor is given by $$Ric=-4b^2g.$$  The metric \eqref{qknew8} seems to be a new explicit quaternionic
K\"ahler metric.

\subsubsection{Explicit quaternionic K\"ahler metrics from the zero-torsion qc-non-flat qc structure on $\mathfrak {l_2}$ }
Consider the simply connected connected Lie group corresponding to the algebra $\mathfrak {l_2}$  defined by the structure equations
\eqref{ex2}.  This group can be described in local coordinates $(x,y,z,t,\varphi,\theta,\psi )$ as follows.
\begin{equation}\label{group L2}
\begin{aligned}
& e_{{1}}= \,-2\,{\it dt}, \quad
e_{{2}}= \,{\it dx}-y{\it dz}+ \left( -2\,x+zy \right) {\it dt},\quad
e_{{3}}= \,{\it dz}-z{\it dt},\quad
e_{{4}}= \,{\it dy}-y{\it dt}\\
& e_{{5}}= \,-2\,{\it dx}+2\,y{\it dz}-2\, \left( -2\,x+zy \right) {\it dt}-4\,\sin \psi \, {\it d{\theta}}
\mbox{}+4\,\cos  \psi  \sin \theta \, {\it d{\varphi}}\\
& e_{{6}}= \,-4\,{\it dz}+4\,z{\it dt}-4\,\cos \psi\, {\it d{\theta}}-4\,\sin \psi \sin  \theta \, {\it d{\varphi}}\\
& e_{{7}}= \,-4\,{\it dy}+4\,y{\it dt}-4\,{\it d{\psi}}-4\,\cos  \theta \, {\it d{\varphi}},
\end{aligned}
\end{equation}
{where $\theta$, $\varphi$, $\psi$ are the Euler angles, $0<\theta < \pi$, $0 < \varphi < 2\pi$ and $0< \psi < 4\pi$.}
Recall, in this case $S=-\frac14$ in \eqref{e:general qK metric}. According to
{\bf (\ref{e:QK subR})},
the corresponding quaternionic  K\"ahler metric on $L_2\times \mathbb{R}$ is
\begin{equation}\label{qknew82}
g=\frac {1}{2b^2} (1+\cosh \sigma)\ g_H+\frac {1}{32b^2} \sinh^2 \sigma\
(\eta_1^2+\eta_2^2+\eta_3^2)+\frac {2}{b^2}\,d\sigma^2.
\end{equation}
The quaternionic K\"ahler 2-forms are expressed by
\begin{equation*}
F_i(\sigma)=\frac {1}{2b^2} (1+\cosh \sigma)\omega_i+\frac {1}{32b^2}
\sinh^2 \sigma\,\eta_j\wedge\eta_k-\frac 12 \eta_i\wedge d\sigma.
\end{equation*}
The Ricci tensor is given by
{\bf $$Ric=-2b^2g.$$}
The metric \eqref{qknew82} seems to be a new explicit quaternionic
K\"ahler metric.

\subsection{{Quaternionic K\"ahler metrics arising from a 3-Sasakian
structure}} Note that here the normalized scalar curvature is $S=2$.  The metric
\begin{equation*}
g= u g_H +(u+au^2)\left ( (\eta_1)^2+(\eta_2)^2+(\eta_3)^2
\right )+ \frac {1}{4(u+au^2)}du^2
\end{equation*}
is a quaternionic K\"ahler, and in the case of $a=0$ is the
hyper K\"ahler cone over the 3-Sasakian manifold. These metrics have
been found earlier in \cite[Theorem 5.2]{WM}.

\subsection{Non quaternionic K\"ahler structures with closed four form in dimension 8}
It is well known \cite{Sw} in dimension $4n$, $n>2$, the
condition that the fundamental 4-form is closed is equivalent to the
fundamental 4-form being parallel which is not true in dimension
eight. Salamon constructed in \cite{Sal} a compact example of
an almost quaternion hermitian manifold with closed fundamental
four form which is not Einstein, and therefore it is not
quaternionic K\"ahler.   We give below explicit  non-compact
example of that kind.

We consider seven dimensional qc structures with zero torsion endomorphism of the Biquard connection and constant qc scalar curvature $S$ satisfying the  structure equations
\begin{equation}\label{e:structure equation for qK}
d\eta_i = 2\omega_i + S\eta_j\wedge\eta_k.
\end{equation}
 Examples of
such manifolds are provided by the following qc manifolds: i) the
quaternionic Heisenberg group, where $S=0$;
 ii) any 3-Sasakian
manifold, where $S=2$ (see \cite{IV1} where it is proved that
these structure equations
characterize the 3-Sasakian qc manifolds); and iii) the zero
torsion qc-flat group $L_1$ defined in Theorem~\ref{m1} with the
structure equations described in \eqref{ex11}, where $S=-1/2$.
Our example are inspired by the following Remark.
\begin{rmrk}\label{r:dim seven QK}
In dimension seven, due to the relations $\omega_i\wedge\omega_j=0$,
$i\not= j$,  a more general evolution than the one used in the
proof of Theorem~\ref{qkmetric} can be considered, namely, let
\begin{equation}\label{genevolution}
\omega_s(t)=f(t)\omega_s, \qquad \eta_s(t)=f_s(t)\eta_s,
\end{equation}
where $f,f_1,f_2,f_3$ are smooth function of $t$. Using the
structure equations \eqref{e:structure equation for qK} one easily
obtains that the equation $d\Omega=0$ is satisfied and
\eqref{qevolunk} is equivalent to the system
\begin{equation}\label{erealqk}
\begin{aligned}
3f'-2(f_1+f_2+f_3)=0,\\
(ff_2f_3)'-S f(f_1-f_2-f_3)-6 f_1f_2f_3=0,\\
(ff_1f_3)'-S f(-f_1+f_2-f_3)-6 f_1f_2f_3=0,\\
(ff_1f_2)'-S f(-f_1-f_2+f_3) -6 f_1f_2f_3=0.
\end{aligned}
\end{equation}
On the other hand, $<F_1, F_2, F_3>$ is a differential ideal  if
and only if the following system holds
\begin{equation}\label{e:closed ideal qk}
f(f_if_j)'-f'f_if_j+2f_1f_2f_3-2f_if_j(f_i+f_j)+S f
f_if_j-S f f_k=0.
\end{equation}
This claim is a consequence of the fact that working mod $ <F_1, F_2,
F_3>$ we have
\[
dF_i=\frac {1}{f}\left
(f(f_if_j)'-f'f_if_j+2f_1f_2f_3-2f_i^2f_j-2f_if_j^2+S
ff_if_j-S f f_k \right )\eta_j\wedge\eta_k\wedge dt.
\]
 Taking
$f_1=f_2=f_3=h$ in \eqref{erealqk} we come to the
case considered in Theorem \ref{qkmetric}.
\end{rmrk}

The system \eqref{erealqk} can be integrated completely when $S=0$.
This is achieved by introducing the new variable $du=f_1f_2f_3dt$,
which allows to determine $ff_if_j=6(u+a_k)$, where $a_k$ is a
constant. Thus
$$f_s=\frac {f}{6(u+a_s)} \frac {du}{dt}.$$ With
the help of these three equations and the first equation of
\eqref{erealqk} we come to $$\frac {9}{f}\frac {df}{dt}=\frac
{du}{dt}\left (\frac {1}{u+a_1} + \frac {1}{u+a_2} + \frac
{1}{u+a_3}\right ),$$ hence $$f^9=C^9(u+a_1)(u+a_2)(u+a_3)$$ for some
constant $C$. Now, the equations $ff_if_j=6(u+a_k)$ and the
definition of $u$ yield
$$f_i=\sqrt{\frac 6C} \left ( \frac {(u+a_j)^4(u+a_k)^4}{(u+a_i)^5} \right )^{1/9}, \quad
dt=\left ( C/6 \right )^{3/2}\frac {du}{\left
((u+a_1)(u+a_2)(u+a_3)\right )^{1/3}}.
$$
when we impose also the
system \eqref{e:closed ideal qk}, in which we substitute
$2(f_i+f_j)=3f'-2f_k$ and $f(f_if_j)'=6f_1f_2f_3-\frac 23 \left (
f_1+f_2+f_3 \right )f_if_j$ (using the equations of
\eqref{erealqk} and $\tau=0$), we see that
\[
3fdF_i={10f_jf_k}\left ( 2f_k-f_i-f_j \right
)\eta_j\wedge\eta_k\wedge dt\qquad mod \ <F_i, F_j, F_k>.
\]
Thus, $d\Phi=0$ and  $<F_1, F_2, F_3>$ is  a differential ideal if
and only if $f_1=f_2=f_3$ which yield the next Theorem.
\begin{thrm}
The  metric  on the product of the seven dimensional quaternionic
Heisenberg group with the  real line defined  by
\begin{multline}\label{e:QK using H general case}
g=C \, \left ((u+a_1)(u+a_2)(u+a_3)\right )^{1/9}\, (dx_1^2+dx_2^2+dx_3^2+dx_4^2) + \\
\frac {6}{C}\left ( \frac {(u+a_2)^8(u+a_3)^8}{(u+a_1)^{10}} \right
)^{1/9}(dx_5+2x_1dx_2+x_3dx_4)^2 +\frac {6}{C}\left ( \frac
{(u+a_3)^8(u+a_1)^8}{(u+a_2)^{10}} \right
)^{1/9}(dx_6+2x_1dx_3+x_4dx_2)^2\\+ \frac {6}{C}\left ( \frac
{(u+a_1)^8(u+a_2)^8}{(u+a_3)^{10}} \right
)^{1/9}(dx_7+2x_1dx_4+x_2dx_3)^2 + \left (\frac {C}{6}\right
)^3\frac {du^2}{\left ((u+a_1)(u+a_2)(u+a_3)\right )^{2/3}},
\end{multline}
where $a_1, a_2$ and $a_3$ are three constants, not all of them equal
to each other, supports { an almost quaternion} hermitian structure
which has closed fundamental form, but is not quaternionic K\"ahler.
\end{thrm}
\begin{rmrk}
Using \textsc{Mathematica} one can check that the metrics
\eqref{e:QK using H general case}  are Einstein exactly when
$f_1=f_2=f_3$, in which case they are quaternionic K\"ahler metrics.
\end{rmrk}

We note that one of the arbitrary constants in \eqref{e:QK using H
general case} is unnecessary since a translation of the  unknown $u$
does not change the metric.

Let us also remark that the quaternionic K\"ahler metric
\eqref{e:general qK metric-f} on the quaternionic Heisenberg group is obtained from the general family \eqref{e:QK
using H general case} by taking $\frac {6}{C^3}=b^2$ and
$v=e^{2bt}=Cu^{1/3}$ when the constants are the same $a_1=a_2=a_3=2b$
and we use $u+a_1$ as a variable, which is denoted also by $u$.

\subsection{Eight dimensional non quaternionic K\"ahler structures with fundamental forms generating a differential ideal}
If one takes a solution of the system \eqref{e:closed ideal qk}
which does not satisfy the system \eqref{erealqk}, one could obtain an almost quaternion
hermitian structure
such that $<F_1, F_2, F_3>$ is a differential ideal, yet, the  structure is not of a
quaternionic K\"{a}hler manifold  
(see also the paragraph after \cite[Corollary 2.4]{Macia}). For example, let $f\equiv 1$ and $S=0$ in the system \eqref{e:closed ideal qk} and introduce  the functions $u_i=\ln (f_jf_k)$. A small calculation turns the system \eqref{e:closed ideal qk} into the following  equivalent system with six equations
\begin{equation}\label{e:diff ideal sys}
f_i=e^{\frac 12 (u_j+u_k-u_i)}, \qquad f_i=\frac 14\left(\frac {du_j}{dt}+\frac {du_k}{dt}\right)
\end{equation}
where as before (and in what follows) $i,\, j,\, k$ denotes a positive permutation of $1,2,3$. Thus we have to solve
\[
\frac {d}{dt}\left( u_j+u_k\right) = 4 e^{\frac 12 (u_j+u_k-u_i)}
\]
Now, we let $v_i=e^{-\frac 12 u_i}$ so the above system becomes $\frac {d}{dt}\left(v_jv_k\right) =-2v_i$. Introducing $w_i=v_jv_k$ the latter equations take the form
\[
\frac {d}{dt}w_i^2 =-4\left( w_1w_2w_3\right)^{1/2}.
\]
Thus, the variables $w_i^2$ defer by additive constants, so we let $x=a_i-w_i^2$, $a_i=$const, where the variable $x$ satisfies $dt=g(x)dx$ with
\begin{equation}\label{e:g nearly quaternionic}
g(x)=\frac 14 \left( (a_1-x)(a_2-x)(a_3-x)\right)^{-1/4}.
\end{equation}
Solving back in terms of the wanted functions $f_i$ we see that
\begin{equation}\label{e:f_i nearly quaternionic}
f_i(x)=\frac {(a_j-x)^{1/4}(a_k-x)^{1/4}}{(a_i-x)^{3/4}}.
\end{equation}
In conclusion, the evolution $\omega_i(x)=\omega_i$, $\eta_i(x)=f_i(x)\eta_i$ leads to the metric
\[
g=g_H +  f_1^2\eta_1^2+f_2^2\eta_2^2+f_3^2\eta_3^2+ g(x)^2dx^2,
\]
where $F_i=\omega_i+f_jf_k\eta_j\wedge \eta_k-gf_i\eta_i\wedge dx$ and the functions $f_i$ and $g$ are defined in \eqref{e:g nearly quaternionic} and \eqref{e:f_i nearly quaternionic} respectively. The above metric supports an almost quaternion hermitian structure such that $<F_1, F_2, F_3>$ is a differential ideal but $g$ is not a  quaternionic K\"{a}hler metric. For the proof of the latter notice that if we set $f=1$ and $S=0$ the system \eqref{erealqk} has no solution with nowhere vanishing functions $f_i$ taking into account the first equation of the system.


\section{$Sp(1)Sp(1)$ structures and $Spin(7)$-holonomy metrics}
An $Sp(1)Sp(1)$ structure on a seven dimensional manifold
$M^7$
%
%
induces a $G_2$-form $\phi$  by
\begin{equation}\label{g2}
\phi=\omega_1\wedge\eta_1+\omega_2\wedge\eta_2+\omega_3\wedge\eta_3-\eta_1\wedge\eta_2\wedge\eta_3.
\end{equation}
Notice that the $G_2$-form \eqref{g2} is $Sp(1)Sp(1)$ invariant hence it is a well defined  global form on the qc-manifold $M$.
The Hodge dual $*\phi$ is
\begin{equation}\label{starg2}
*\phi=\frac{1}{2}\omega_1^2-(\omega_1\wedge\eta_2\wedge\eta_3+
\omega_2\wedge\eta_3\wedge\eta_1+\omega_3\wedge\eta_1\wedge\eta_2).
\end{equation}
Consider the family $(\eta_s(u),\omega_s(u))$ of $Sp(1)Sp(1)$ structures on $M^7$ depending on real parameter $u$, the corresponding $G_2$ form $\phi(u)$ and the $Spin(7)$-form $\Psi(u)$ on $M^7\times\mathbb R$ defined
by \cite{BH}
\begin{equation}\label{spin7}
\Psi(u)=F^-_1\wedge F^-_1+F^-_2\wedge F^-_2-F^-_3\wedge F^-_3=2\Big[*\phi(u)-\phi(u)\wedge du\Big],
\end{equation}
where the 2-forms $F^-_1,F^-_2,F^-_3$ are given  by
\begin{equation}\label{qhyp1}
\begin{aligned}
F^-_1=\omega_1-\eta_2\wedge\eta_3-\eta_1\wedge du,\quad
F^-_2=\omega_2-\eta_3\wedge\eta_1-\eta_2\wedge du,\quad
F^-_3=\omega_3+\eta_1\wedge\eta_2+\eta_3\wedge du.
\end{aligned}
\end{equation}
Following Hitchin, \cite{Hitt}, the $Spin(7)$-form $\Psi(u)$ is closed
{ (and so \cite{Fernandez} parallel with respect to the Levi-Civita connection)}
if and only if the $G_2$ structure is cocalibrated, $d*\phi=0$, and
the Hitchin flow equations $\partial_u(*\phi)=-d\phi$ are satisfied together with initial conditions $d*(\phi(u_0))=0$ at some point $u_0$,
i.e.
\begin{equation}\label{evolspin}
\partial_u(*\phi)=-d\phi, \qquad d*(\phi(u_0))=0.
\end{equation}

\subsection{Construction of $Spin(7)$-holonomy metrics using qc structures}
At this point we shall assume that $(M,g,\mathbb Q)$ is a qc manifold  of
dimension seven and investigate Hitchin's equations \eqref{evolspin} leading to $Spin(7)$-holonomy metrics.

Recall that in dimension seven the fundamental four form of the qc structure is given by $\Omega=3\omega_1\wedge\omega_1$.
The latter together with Lemma~\ref{closedqc}, \eqref{starg2} and Theorem~\ref{integr-closed} yield the next Theorem.
\begin{thrm}\label{integr-g2}
Let $(M,g,\mathbb Q)$ be seven dimensional qc manifold.  The following conditions are equivalent.
\begin{itemize}
\item[a).] The fundamental four form is closed, $d\Omega=0$;
\item[b).] The $G_2$-structure \eqref{g2} is cocalibrated.
\item[c).] The vertical distribution is integrable;
\end{itemize}
\end{thrm}

\begin{cor}\label{work}
The $G_2$-structure \eqref{g2} induced from a qc Einstein structure with constant qc scalar curvature is co-calibrated.
\end{cor}
\begin{proof}
The assumptions of the corollary lead to  the structure equations \eqref{strom1} which imply $d\Omega=0$ since in dimension seven $\omega_s\wedge\omega_t=\delta_{st}vol.|_H$. Now, Theorem~\ref{integr-g2} completes the proof.
\end{proof}

Earlier \cite{GS} and \cite{FKMS} observed that every 3-Sasakian manifold has $G_2$-forms, which are nearly parallel, and each one of them has a
{ ``squashing"}
which produces another nearly parallel $G_2$-form. These structures are then used to obtain $Spin(7)$ metrics on the metric cone \cite{B}. In \cite{AF} it is proven that every 3-Sasakian manifold has a ``canonical" $G_2$-form which is co-calibrated. From Corollary \ref{work}, a seven dimensional qc Einstein structure with constant qc scalar curvature has a co-calibrated $G_2$-form which by \cite{Hitt} is a good candidate to construct a metric with holonomy contained in $Spin(7)$. It should be pointed out that a seven dimensional qc-Einstein structure with positive qc constant scalar curvature, locally, has a 3-Sasakian structure, see \cite{IMV}.   Nevertheless, the squashed metrics mentioned above are examples of seven dimensional qc Einstein structure with positive constant qc scalar curvature, so the quaternionic contact point of view allows a unified treatment of the construction.

 We turn to the main result allowing the construction of  $Spin(7)$-holonomy metrics on the product of a qc manifold with a real interval.
\begin{thrm}\label{spin7metric}
Let $(M,g,\mathbb Q)$ be a smooth qc Einstein manifold of dimension seven with constant normalized qc scalar curvature $S$.  For a suitable constant $a$, the manifold $M\times\mathbb{R}$ has a
metric with holonomy contained in $Spin(7)$  given by the following metric and $Spin(7)$-form
\begin{equation}\label{e:general Spin(7) metric}
\begin{aligned}
&g=ug_H+\frac{S u^{5/3}-2a}{10u^{2/3}}\left (
(\eta_1)^2+(\eta_2)^2+(\eta_3)^2 \right )+ \frac
{5u^{2/3}}{18 (S u^{5/3}-2a)}du^2,\\
&\Psi=F^-_1\wedge F^-_1 + F^-_2\wedge F^-_2 - F^-_3\wedge F^-_3,
\end{aligned}
\end{equation}
where
\begin{equation}\label{e:local spin(7) forms}
\begin{aligned}
& F^-_1(u)=u\omega_1-\frac {S
u^{5/3}-2a}{10u^{2/3}}\,\eta_2\wedge\eta_3- \frac
16\,\eta_1\wedge du, \\
&F^-_2(u)=u\omega_2-\frac {S
u^{5/3}-2a}{10u^{2/3}}\,\eta_3\wedge\eta_1-\frac
16\,\eta_2\wedge du, \\
& F^-_3(u)=u\omega_3+\frac {S
u^{5/3}-2a}{10u^{2/3}}\,\eta_1\wedge\eta_2+\frac
16\,\eta_3\wedge du.
\end{aligned}
\end{equation}
\end{thrm}
\begin{proof}
Let $h$ and $f$ be some functions of the unknown $t$ . Consider the 2-forms defined by
\begin{equation}\label{fff7}
\begin{aligned}
&F^-_1(t)=f(t)\omega_1-h^2(t)\eta_2\wedge \eta_3-h(t)\eta_1\wedge dt,\\
&F^-_2(t)=f(t)\omega_2-h^2(t)\eta_3\wedge \eta_1-h(t)\eta_2\wedge dt,\\
&F^-_3(t)=f(t)\omega_i+h^2(t)\eta_1\wedge \eta_2+h(t)\eta_3\wedge dt.
\end{aligned}
\end{equation}
Substituting \eqref{fff7} in \eqref{spin7}, then taking the exterior derivative of the obtained 4-form while applying \eqref{streq} and \eqref{strom1} yields
\begin{multline}\label{sp7}
d\Psi(t)=(2ff'-12fh)\omega_1\wedge\omega_1\wedge dt\\-\Big(2(fh^2)'-2fhS-4h^3\Big)(\omega_1\wedge\eta_2\wedge\eta_3+\omega_2\wedge\eta_3\wedge\eta_1
+\omega_3\wedge\eta_1\wedge\eta_2)\wedge dt.
\end{multline}
Hence, \eqref{sp7} shows that the condition $d\Psi(t)=0$ is equivalent  to the ODE system
\begin{equation}\label{sol7}
3ff''+(f')^2-9S f=0, \qquad h=\frac16 f'.
\end{equation}
To solve this differential equation, we use  $v=f^{4/3}$ as a
variable. Equation \eqref{sol7} shows  that $$\left ( \frac
{dv}{dt}\right )^2= \frac{32(S v^{5/4}-2a)}{5},$$ where $a$ is a
constant. Hence, $\left( \frac {dt}{df} \right)^2 =\left( \frac
{dt}{dv} \right )^2 \left( \frac {dv}{df}\right )^2=\frac
{5f^{2/3}}{18(S f^{5/3}-2a)}$, which implies
$$h^2=\frac{1}{36} (f')^2 =\frac{S
f^{5/3}-2a}{10f^{2/3}}.$$ Renaming $f$ to $u$ gives the metric
and the $Spin(7)$ form $\psi$ in the form \eqref{e:general Spin(7) metric}
together with \eqref{e:local spin(7) forms}.
\end{proof}

\subsection{$Spin(7)$-holonomy metrics based on qc Einstein structure with zero qc scalar curvature}

\subsubsection{$Spin(7)$ holonomy metrics from the quaternionic
Heisenberg group} Consider the 7-dimensional
quaternionic Heisenberg group $\boldsymbol{G\,(\mathbb{H})}$ with structure equations \eqref{4n+3heis} taken for $n=1$ equipped with its standard qc structure. The corresponding eight dimensional $Spin(7)$-holonomy
metric written in Theorem~\ref{spin7metric} can be written in the form
\begin{equation}\label{citsp7}
\begin{aligned}
g=u^3\left ((e^1)^2+(e^2)^2+(e^3)^2+(e^4)^2 \right )+\frac
{a^2}{16}u^{-2}\left ( (\eta_1)^2+(\eta_2)^2+(\eta_3)^2 \right )+
\frac {4}{a^2}u^{6}du^2.
\end{aligned}
\end{equation}
These $Spin(7)$-holonomy metrics are found in \cite[Section
4.3.1]{GLPS}.

\subsubsection{New $Spin(7)$-holonomy metrics from the quaternionic Heisenberg group}
New $Spin(7)$-holonomy metrics can be obtained similarly to the
derivation of \eqref{e:QK using H general case}. We evolve the structure as in \eqref{genevolution}, namely
$\omega_s(u)=f(u)\omega_s, \quad \eta_s(u)=f_s(u)\eta_s$.  Using the
structure equations of the quaternionic Heisenbrg group, $d\eta_s=2\omega_s$, one easily obtains that the second
equation of the  \eqref{evolspin} is equivalent to the system
\begin{equation}\label{ereal7}
\begin{aligned}
f'-2(f_1+f_2+f_3)=0,\quad
(ff_2f_3)'-2 f_1f_2f_3=0,\\
(ff_1f_3)'-2 f_1f_2f_3=0,\quad
(ff_1f_2)'-2 f_1f_2f_3=0.
\end{aligned}
\end{equation}
We integrate the system \eqref{ereal7} to obtain the
next family of $Spin(7)$-holonomy metrics   which seems to be new
\begin{multline}\label{e:Spin(7) using H general case}
g=C \, \left ((u+a_1)(u+a_2)(a_3-u)\right )\, \left
((dx_1^2+dx_2^2+dx_3^2+dx_4^2 \right ) \\+ \frac {2}{C} \frac
{1}{(u+a_1)^{2}}(dx_5+2x_1dx_2+2x_3dx_4)^2 + \frac2C\frac
{1}{(u+a_2)^{2}} (dx_6+2x_1dx_3+2x_4dx_2)^2+ \\\frac2C\frac
{1}{(a_3-u)^{2}}(dx_7+2x_1dx_4+2x_2dx_3)^2
 + \frac {C^3}{8}(u+a_1)^2(u+a_2)^2(a_3-u)^2du^2.
\end{multline}
Taking $a_2=-a_3=a_1$ into \eqref{e:Spin(7) using H general case}
one gets the Spin(7)-holonomy metrics \eqref{citsp7}. Since the
coefficients of the metrics \eqref{e:Spin(7) using H general case}
are continuous with respect to the parameters, and since the
holonomy is equal to Spin(7) for $(a_2,a_3)=(a_1,-a_1)$ then the
same holds for any $(a_2,a_3)$ in an small neighbourhood of
$(a_1,-a_1)$. Thus, we get a three parameter family of metrics
with holonomy equal to $Spin(7)$ which seem to be new.

More generally, for any triple $a_1,a_2,a_3$ of real numbers one
can find an open interval $J\subset\mathbb R$ such that for $u\in
J$ the holonomy of the metrics \eqref{e:Spin(7) using H general
case} equals $Spin(7)$.

\subsection{$Spin(7)$-holonomy metrics based on qc Einstein structure with negative scalar curvature}

\subsubsection{Explicit $Spin(7)$-holonomy metrics from the zero torsion qc-flat structure on $\mathfrak {l_1}$}

 We consider the Lie group $ {L_1}$ defined by the structure equations \eqref{ex11} which can be described in local coordinates with \eqref{loccoord}. In this case $S=-\frac12$ according to Theorem~\ref{m1}. The corresponding metric with holonomy contained in $Spin(7)$ from Theorem~\ref{spin7metric} has the form { (taking $b=-4a$)}
\begin{equation}\label{newspin7metric}
g=u\left ((e^1)^2+(e^2)^2+(e^3)^2+(e^4)^2 \right )+
\frac{(b-u^{5/3})}{20u^{2/3}}\left (
(\eta_1)^2+(\eta_2)^2+(\eta_3)^2 \right )+ \frac
{5u^{2/3}}{9(b-u^{5/3})}du^2.
\end{equation}
and seem to be new.

Calculating the curvature of the metrics \eqref{newspin7metric}, using the local coordinates \eqref{loccoord}
of the group, one finds that there are at least 16 independent curvature forms which
implies that the holonomy of these metrics is equal to Spin(7).

\subsubsection{Explicit $Spin(7)$-holonomy metrics from the zero torsion qc-non-flat structure on $\mathfrak {l_2}$} We consider the Lie group $ {L_2}$  defined by the structure equations \eqref{ex2} which can be described in local coordinates with \eqref{group L2}. In this case $S=-\frac14$ according to Theorem~\ref{m22}. The corresponding metric with holonomy contained in $Spin(7)$ from Theorem~\ref{spin7metric} takes the form given by equation \eqref{e:general Spin(7) metric} with $S=-\frac14$ { and $b=-8a$,
\begin{equation}\label{spin7l2}
g=u\left ((e^1)^2+(e^2)^2+(e^3)^2+(e^4)^2 \right )+
\frac{(b-u^{5/3})}{40u^{2/3}}\left (
(\eta_1)^2+(\eta_2)^2+(\eta_3)^2 \right )+ \frac
{10u^{2/3}}{9(b-u^{5/3})}du^2.
\end{equation}
and seems to be  a new metric with holonomy equal to $Spin(7)$.  The latter fact can be seen by a direct calculation applying the local coordinates \eqref{group L2} to \eqref{spin7l2} and showing that the curvature 2-forms span a space of dimension twenty one.}

\subsection{$Spin(7)$-holonomy metrics from a 3-Sasakian manifold}
This case was investigated in general in \cite{Baz} and explicit
solutions in particular cases are known (see \cite{Baz} and
references therein). We use again only the particular solution to
\eqref{ereal7} found above.  
 The  metric with holonomy contained in $Spin(7)$ from
Theorem~\ref{spin7metric} takes the form given by equation
\eqref{e:general Spin(7) metric} with $S=2$,
\begin{equation*}
g=u\left ((e^1)^2+(e^2)^2+(e^3)^2+(e^4)^2 \right )+\frac
{u^{5/3}-a}{5u^{2/3}}\left ( (\eta_1)^2+(\eta_2)^2+(\eta_3)^2 \right
)+ \frac {5u^{2/3}}{36(u^{5/3}-a)}du^2.
\end{equation*}
{This family includes the (first) complete metric with holonomy $Spin(7)$
constructed by Bryant and Salamon on the total space of the spin bundle over the sphere $S^4$~\cite{BrS,GPP}, see also \cite[p. 6]{Baz}.
}

\end{document}